\documentclass[oneside,reqno,11pt]{amsart}
\usepackage[utf8]{inputenc}
\usepackage{geometry}
\usepackage{hyperref}       
\usepackage{url}            
\usepackage{microtype}      
\usepackage{amsfonts}
\usepackage{amsthm}
\usepackage{amsmath}
\usepackage{amssymb}
\usepackage{bm}
\usepackage{mdframed}
\usepackage{dsfont}
\usepackage{soul}
\usepackage{hyperref}
\usepackage{todonotes}

\newcommand{\sca}[2]{\langle #1 | #2\rangle}
\newcommand{\nr}[1]{\left\Vert #1\right\Vert}
\newcommand{\abs}[1]{\left\vert #1\right\vert}

\newcommand{\Rsp}{\mathbb{R}}

\newcommand{\id}{\mathrm{id}}
\newcommand{\DD}{\mathrm{D}}
\newcommand{\Class}{\mathcal{C}}
\newcommand{\Wass}{\mathrm{W}}

\newcommand{\X}{\mathcal{X}}

\newcommand{\dd}{\mathrm{d}}

\renewcommand{\L}{\mathrm{L}}

\newcommand{\diam}{\mathrm{diam}}

\newcommand{\eps}{\varepsilon}
\newcommand{\TV}{\mathrm{TV}}

\newcommand{\Prob}{\mathcal{P}}

\newcommand{\Var}{\mathbb{V}\mathrm{ar}}

\newcommand{\interior}{\mathrm{int}}
\newcommand{\PP}{\mathbb{P}}
\newcommand{\QQ}{\mathbb{Q}}
\newcommand{\Esp}{\mathbb{E}}
\newcommand{\Kant}{\mathcal{K}}
\newcommand{\Primal}{\mathrm{(P)}_\PP}
\newcommand{\Dual}{\mathrm{(D)_\PP}}

\newcommand{\spt}{\mathrm{spt}}

\newcommand{\WWass}{\mathcal{W}}
\newcommand{\Haus}{\mathcal{H}}
\newcommand{\per}{\mathrm{per}}

\newmdtheoremenv{theo}{Theorem}
\newmdtheoremenv{coro}{Corollary}

\newtheorem{theorem}{Theorem}[section]
\newtheorem*{theoremst}{Theorem}
\newtheorem{corollary}[theorem]{Corollary}

\newtheorem{proposition}[theorem]{Proposition}
\newtheorem*{propositionst}{Proposition}
\newtheorem{assumption}[theorem]{Assumption}

\theoremstyle{definition}

\usepackage{soul}

\newtheorem{remark}{Remark}[section]
\newtheorem{example}[theorem]{Example}

\title[Quantitative Stability of Barycenters in the Wasserstein space]{Quantitative Stability of Barycenters \\in the Wasserstein space}

\author{Guillaume Carlier}
\address{Ceremade, Univ. Paris-Dauphine PSL, 75775 Paris \and Mokaplan, Inria Paris}
\email{carlier@ceremade.dauphine.fr}

\author{Alex Delalande}
\address{Lagrange Mathematics and Computing Research Center, 75007, Paris, France}
\email{delalande.alex@gmail.com}

\author{Quentin Mérigot}
\address{Université Paris-Saclay, CNRS, Laboratoire de mathématiques d’Orsay, 91405, Orsay, France \and Institut universitaire de France}
\email{quentin.merigot@universite-paris-saclay.fr}

\begin{document}

\maketitle

\begin{abstract}
   Wasserstein barycenters define averages of probability measures in a geometrically meaningful way. Their use is increasingly popular in applied fields, such as image, geometry or language processing. In these fields however, the probability measures of interest are often not accessible in their entirety and the practitioner may have to deal with statistical or computational approximations instead. In this article, we quantify the effect of such approximations on the corresponding barycenters. We show that Wasserstein barycenters depend in a Hölder-continuous way on their marginals under relatively mild assumptions. Our proof relies on recent estimates that allow to quantify the strong convexity of the barycenter functional. Consequences regarding the statistical estimation of Wasserstein barycenters and the convergence of regularized Wasserstein barycenters towards their non-regularized counterparts are explored. 
\end{abstract} 

\vspace{0.1cm}


\smallskip


\section{Introduction}
\label{sec:introduction}

Wasserstein barycenters are Fréchet means in Wasserstein spaces: they define averages of families of probability measures that are consistent with the optimal transport geometry and generalize to more than two measures the fundamental notion of displacement interpolation due to McCann \cite{McCann}. As such, they average out probability measures in a geometrically meaningful way and appear as a relevant tool to interpolate or summarize measure data. This notion of barycenter have indeed found many successful applications, for instance in image processing \cite{image_processing}, geometry processing \cite{geometry_processing}, language processing \cite{language_processing_1, language_processing_2, language_processing_3}, statistics \cite{JMLR:v19:17-084} or machine learning \cite{pmlr-v32-cuturi14, pmlr-v70-ho17a}. We refer the readers to existing surveys \cite{comp_OT, stats_wass_space} for further applications. In such applications however, the probability measures of interest are often not accessible in their entirety. They may be accessible for instance only through noisy samples in a statistical context, or they may be approximated in order to use existing computational methods that estimate Wasserstein barycenters (see e.g. \cite{refId0, IBP, pmlr-v32-cuturi14, JMLR:v22:20-588}) while paying an affordable computational cost. Thus, in addition to the computational error induced by the algorithm used to calculate the barycenter, the practitioner may be subject to an extra statistical or approximation error that corresponds to the approximation of the marginal measures of interest. While works focusing on the computation of Wasserstein barycenters may now come with guarantees on the first type of error (see e.g. \cite{JMLR:v22:20-588}), very little is known on the second type of error, which corresponds broadly speaking to a stability error since it quantifies the effect of a perturbation of the marginals on the corresponding barycenters. In this work, we focus on this type of error and show that the Wasserstein barycenter depends in an Hölder-continuous way on its marginal measures under regularity assumptions on (some of) the latter. In the remainder of this section, we define Wasserstein barycenters and the setting we focus on. We then show that mild regularity assumptions are necessary in order to hope for any stability result. Next, we give the dual formulation of the Wasserstein barycenter problem in our context, that is necessary to present our main assumption. This assumption and our main result are then stated and we conclude this section by giving some immediate but useful consequences of our main result.

\subsection{Wasserstein barycenters} Introduced in \cite{agueh:hal-00637399} for finite families of probability measures supported over a Euclidean space, the definition of Wasserstein barycenters have been extended to infinite families of probability measures in \cite{bigot-klein, PASS2013947}, possibly supported over a Riemannian manifold in \cite{KIM2017640, LeGouic2017}. In this work, we focus on families of probability measures supported over a compact Euclidean domain. Let $\Omega = B(0, R) \subset \Rsp^d$ be the ball of $\Rsp^d$ centered at zero and of radius $R > 0$ and denote $\Prob(\Omega)$ the set of Borel probability measures over $\Omega$. We endow $\Prob(\Omega)$ with the $2$-Wasserstein distance $\Wass_2$ defined for any $\rho, \mu \in \Prob(\Omega)$ by
\begin{equation*}
    \Wass_2(\rho,  \mu) = \left( \min_{\gamma \in \Gamma(\rho, \mu)} \int_{\Omega \times \Omega} \nr{x-y}^2 \dd \gamma(x,y) \right)^{1/2},
\end{equation*}
where the minimum is taken over the set $\Gamma(\rho, \mu)$ of transport plans between $\rho$ and $\mu$. 
We equip $\Prob(\Omega)$ with the 
the topology induced by $\Wass_2$ (i.e. the weak topology) and denote $\Prob(\Prob(\Omega))$ the set of corresponding Borel probability measures over $\Prob(\Omega)$. For a measure $\PP \in \Prob(\Prob(\Omega))$, we introduce its \emph{variance functional} $F_\PP$ defined from $\Prob(\Omega)$ to $\Rsp$ by:
\begin{align*}
F_\PP : \mu \mapsto \frac{1}{2} \int_{\Prob(\Omega)} \Wass_2^2(\rho, \mu) \dd \PP(\rho).
\end{align*}
A Wasserstein barycenter of $\PP \in \Prob(\Prob(\Omega))$ is then defined as a minimizer $\mu_\PP$ of the variance functional $F_\PP$:
\begin{equation*}
    \mu_\PP \in \arg \min_{\mu \in \Prob(\Omega)}  F_\PP(\mu).
\end{equation*}
\noindent Such a minimizer always exists, 
and it is uniquely defined whenever $\PP(\Prob_{a.c.}(\Omega))>0$, 
where $\Prob_{a.c.}(\Omega)$ denotes the set of probability measures over $\Omega$ that are absolutely continuous with respect to the Lebesgue measure \cite{KIM2017640, LeGouic2017}. 

\subsection{Stability of Wasserstein barycenters}
As mentioned above, the population of interest $\PP \in \Prob(\Prob(\Omega))$ may not always be accessible in practice, and one may have to deal with another measure $\QQ \in \Prob(\Prob(\Omega))$ instead. The stability question that then comes up is the following: can we bound a distance between minimizers $\mu_{\PP}$ of $F_\PP$ and $\mu_{\QQ}$ of $F_\QQ$ in terms of a distance between $\PP$ and $\QQ$? While the above-defined $2$-Wasserstein distance gives a natural metric to compare $\mu_\PP$ and $\mu_\QQ$, there remains to choose a metric in order to compare  $\PP$ and $\QQ$. For this, we will use the following $1$-Wasserstein distance over $\Prob(\Prob(\Omega))$, defined for any $\PP, \QQ$ in $\Prob(\Prob(\Omega))$ by
\begin{equation*}
    \WWass_1(\PP, \QQ) = \min_{\gamma \in \Gamma(\PP, \QQ)} \int_{\Prob(\Omega) \times \Prob(\Omega)} \Wass_2(\rho, \tilde{\rho}) \dd \gamma(\rho, \tilde{\rho}).
\end{equation*}
This choice of distance is justified by the fact that Wasserstein distances are naturally defined for probability measures on the compact metric space $(\Prob(\Omega), \Wass_2)$ and that they allow to compare measures that have incomparable support. The $1$-Wasserstein distance being the weakest of the Wasserstein distances, our bounds are ensured to be the sharpest in terms of this optimal transport geometry. We are thus interested in bounding $\Wass_2(\mu_\PP, \mu_\QQ)$ in terms of $\WWass_1(\PP, \QQ)$ for $\PP, \QQ \in \Prob(\Prob(\Omega))$.

\subsubsection{Consistency of Wasserstein barycenters.} Before looking for any quantitative stability result, one may first wonder if the Wasserstein barycenters depend at least in a continuous way on their marginals. This question, framed under the notion of \emph{consistency} of Wasserstein barycenters, has been answered positively in \cite{bigot-klein, boissard-legouic-loubes} in some specific settings and in \cite{LeGouic2017} in the most general setting. Theorem 3 of \cite{LeGouic2017} ensures in particular the following: 

\begin{theoremst}[Le Gouic, Loubes]
Let $\PP \in \Prob(\Prob(\Omega))$ and a sequence $(\PP_n)_{n \geq 1} \in \Prob(\Prob(\Omega))$ be such that $$\WWass_1(\PP_n, \PP) \xrightarrow[n \to + \infty]{} 0.$$ For all $n \geq 1$, denote $\mu_{\PP_n}$ a barycenter of $\PP_n$. Then the sequence $(\mu_{\PP_n})_{n \geq 1}$ is precompact in $(\Prob(\Omega), \Wass_2)$ and any limit is a barycenter of $\PP$.  
\end{theoremst} 
\noindent This result ensures the continuity of Wasserstein barycenters with respect to the marginal measures, at least in our setting, so that we can now legitimately look for bounds that quantify this continuity. 

\subsubsection{Quantitative stability in dimension $d=1$.} \label{sec:stab-dim-1} In dimension $d = 1$, the derivation of quantitative stability bounds for Wasserstein barycenters is straightforward. Indeed, in this context $\Wass_2$ is Hilbertian, which ensures a Lipschitz behavior of the barycenters with respect to their marginals. More precisely, denoting $Q_\rho$ the \emph{quantile function} of a measure $\rho \in \Prob(\Omega)$ (i.e. the generalized inverse of its cumulative distribution function), one has for any measures $\rho, \mu \in \Prob(\Omega)$ that $\Wass_2(\rho, \mu) = \nr{Q_\rho - Q_\mu}_{\L^2([0,1])}$. This leads for any $\PP \in \Prob(\Prob(\Omega))$ to a simple formula for the unique barycenter:
\begin{equation*}
    \mu_\PP = \left( \int_{\Prob(\Omega)} Q_\rho \dd \PP(\rho) \right)_{\#} \lambda_{[0, 1]},
\end{equation*} 
where $\lambda_{[0, 1]}$ denotes the Lebesgue measure over $[0, 1]$. Using this fact and the triangle inequality, one immediately obtains the following Lipschitz stability result, that actually holds for any families of measures in the set $\Prob_2(\Rsp)$ of probability measures supported over $\Rsp$ that admit a finite second-order moment:
\begin{propositionst}
Let $\PP, \QQ \in \Prob(\Prob_2(\Rsp))$ and denote $\mu_\PP, \mu_\QQ$ their respective barycenters. Then
$$ \Wass_2(\mu_\PP, \mu_\QQ) \leq \WWass_1(\PP, \QQ). $$
\end{propositionst}
\noindent This fact was exploited in \cite{wass_bar_1d_stats} to characterize the statistical rate of convergence of empirical Wasserstein barycenters towards their population counterpart in an asymptotic setting for probability measures supported over the real line. 

\begin{figure}
    \centering
    \includegraphics[width=.34\textwidth]{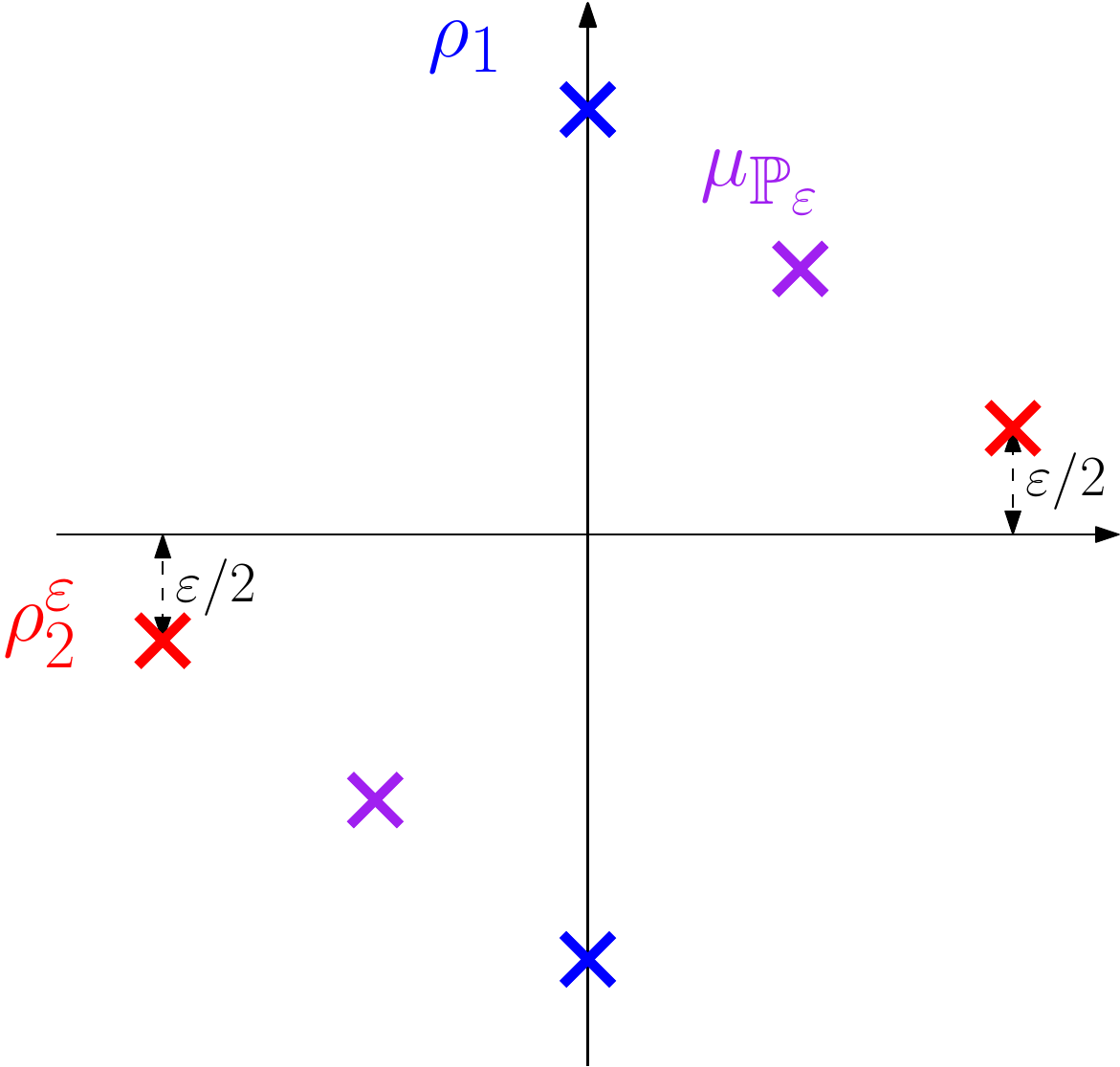}  \hspace{1.5cm}
    \includegraphics[width=.34\textwidth]{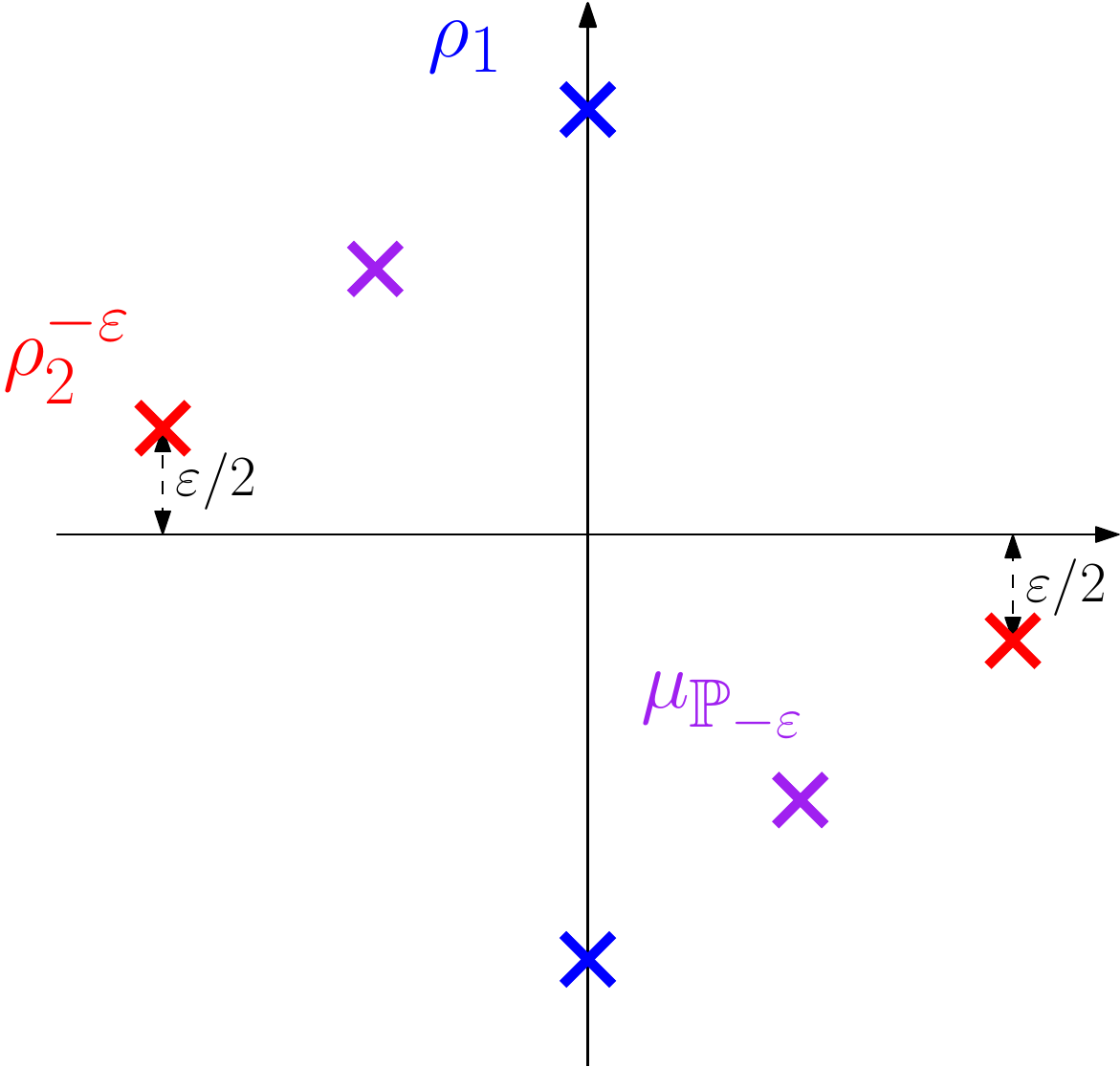} 
    \caption{Let $\rho_1 = \frac{1}{2}(\delta_{(0; 1)} + \delta_{(0; -1)})$. For $\eps > 0$ and $x_\eps = (1; \eps/2) \in \Rsp^2$, let $\rho_2^\eps = \frac{1}{2}(\delta_{x^\eps} + \delta_{-x^\eps})$. Introduce $\PP_\eps = \frac{1}{2} (\delta_{\rho_1} + \delta_{\rho_2^\eps})$. Then for $\eps\leq \frac{1}{2}$, $\Wass_2(\mu_{\PP_\eps}, \mu_{\PP_{-\eps}}) = 1$ while $\WWass_1(\PP_\eps, \PP_{-\eps}) \leq \eps$.}
    \label{fig:no-stability}
\end{figure}

\subsubsection{Quantitative stability in dimension $d \geq 2$.} In dimension $d \geq 2$, the derivation of any quantitative stability bound turns out to be much more difficult. This first comes from the fact that without assumption on $\PP$ and $\QQ$, the barycenters $\mu_\PP$ and $\mu_\QQ$ may not be uniquely defined, which makes hopeless the derivation of any stability result. Even when uniqueness of the barycenters is ensured, one can easily build examples where no quantitative stability bound holds, see for instance the setting illustrated in Figure \ref{fig:no-stability}. This example relies on barycenters with only discrete marginals, and recovers in the limit $\eps = 0$ the pathological case where the barycenter is not uniquely defined. One may circumvent this issue by ensuring, even in the limit $\eps=0$, uniqueness of the barycenter. As mentioned above, this can be done by imposing that some of the marginal measures are absolutely continuous. Nevertheless, even under such an assumption on the marginals, one can easily build an example where the barycenter achieves an Hölder behavior with respect to its marginal, but with an Hölder exponent that can be chosen arbitrarily small, see Figure \ref{fig:holder-bad}. These negative results show that, even in dimension $d=2$, regularity assumptions on the marginals $\PP, \QQ$ that go beyond sole absolute continuity are necessary in order to hope to derive stability estimates for their barycenters.

\begin{figure}
    \centering
    \includegraphics[width=.34\textwidth]{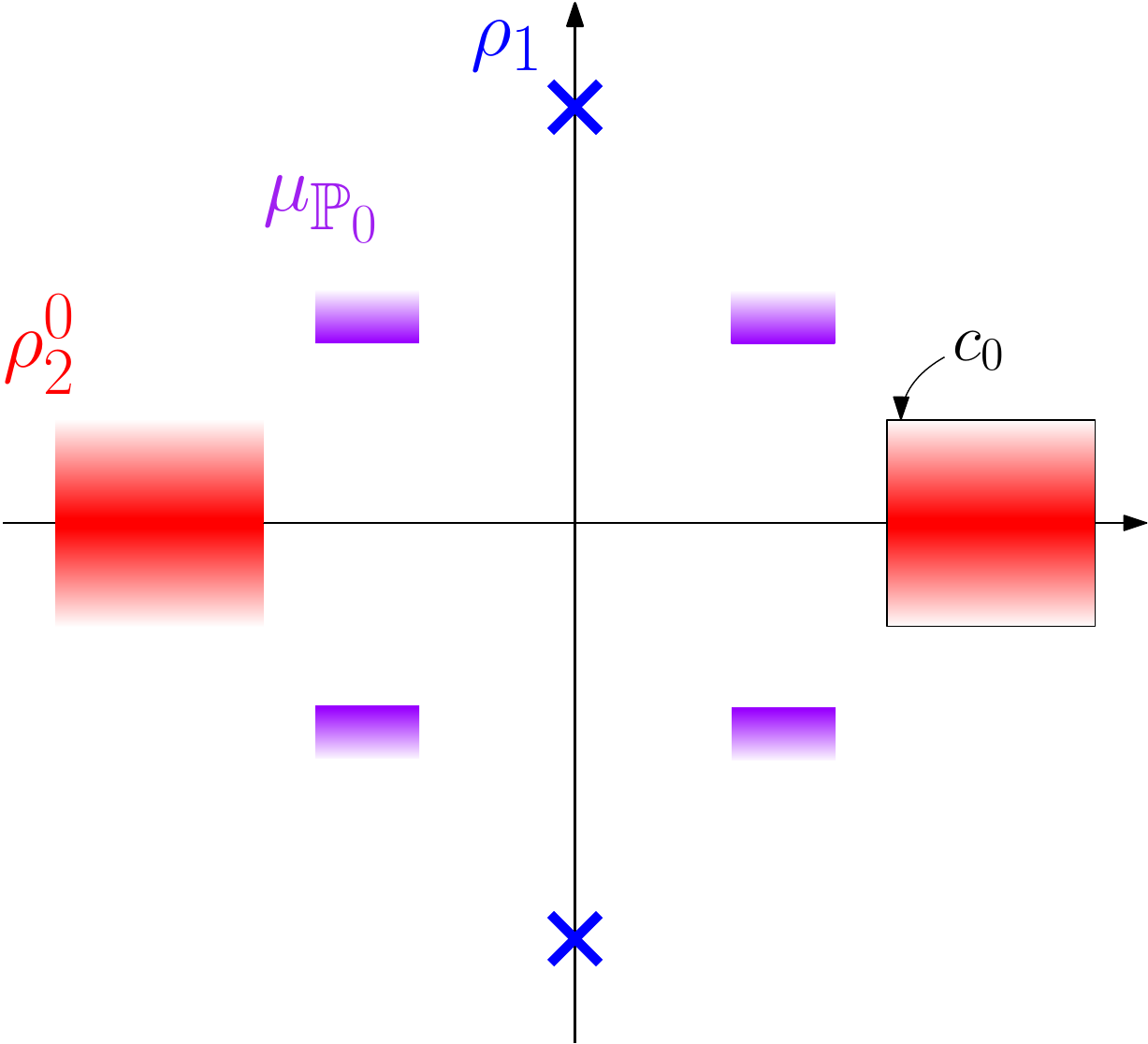}  \hspace{1.5cm}
    \includegraphics[width=.34\textwidth]{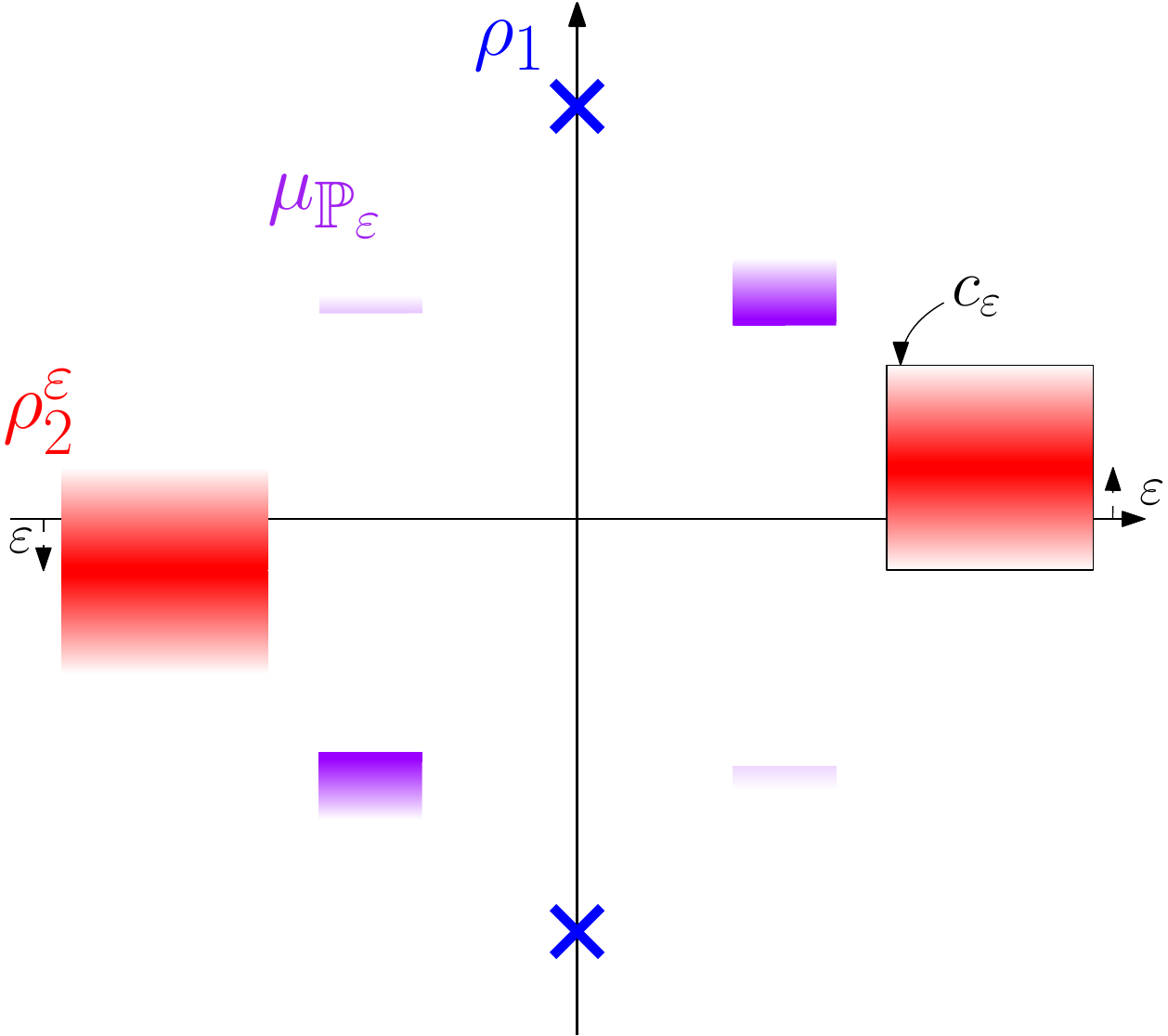} 
    \caption{Let $\rho_1 = \frac{1}{2}(\delta_{(0; 1)} + \delta_{(0; -1)})$. For $a \in (0, 1)$ and $\eps > 0$, let $c_\eps = [1-\frac{a}{2}; 1 + \frac{a}{2}] \times [-\frac{a}{2} + \eps; \frac{a}{2} + \eps]$ and $\rho_2^\eps$ the probability measure with density $\rho_2^\eps(x,y) = \frac{\alpha}{2^{1-2\alpha} a ^{1 + 2\alpha}}\left( \abs{y-\eps}^{2\alpha-1} \mathds{1}_{c_\eps}(x,y) + \abs{y+\eps}^{2\alpha-1} \mathds{1}_{-c_\eps}(x,y) \right)$ for some $\alpha > 0$. Introduce $\PP_\eps = \frac{1}{2} (\delta_{\rho_1} + \delta_{\rho_2^\eps})$. Then for $\eps \leq \frac{a}{2}$, $\Wass_2(\mu_{\PP_0}, \mu_{\PP_\eps}) \sim  \eps^{\alpha}$ while $\WWass_1(\PP_0, \PP_\eps) \leq \eps$.}
    \label{fig:holder-bad}
\end{figure}

\subsubsection{Previous works.} Consistently with the above remarks, previous works having dealt with the stability of Wasserstein barycenters have either worked under stringent assumptions on the marginal measures or regularized the barycenter problem in order to ensure more regular solutions. In \cite{Ahidar-Coutrix2020, LeGouic2022FastCO} for instance, the question of the rate of convergence of the empirical barycenter in a Wasserstein space towards its population counterpart has been answered at the cost of assumptions that require in particular to have guarantees on the regularity of the (unknown) population barycenter (see sub-section \ref{subsec:cv-rate-emp-bar} for more details). In \cite{bigot:hal-01564007, entropic-barycenters}, a regularization of the barycenter problem has been considered and stability bounds and central limit theorems were deduced for the solutions to this regularized problem. In this work, we do not regularize the variance functional and work under less restrictive assumptions on the marginal measures than previous works having dealt with the stability of Wasserstein barycenters. In order to state these assumptions, we first need to introduce the dual problem to the Wasserstein barycenter problem.

\subsection{Dual formulation}
Building from \cite{agueh:hal-00637399}, we show that the Wasserstein barycenter problem admits the following dual formulation with strong duality. The proof of this proposition is deferred to the appendix, Section \ref{sec:dual-formulation}.
\begin{proposition}[Dual formulation]
\label{prop:dual-formulation}
For any $\PP \in \Prob(\Prob(\Omega))$, one has
\begin{equation*}
    \min_{\mu \in \Prob(\Omega)} F_\PP(\mu) = \frac{1}{2} \int_{\Prob(\Omega)}  M_2(\rho) \dd \PP(\rho) - \Dual,
\end{equation*}
where $M_2(\rho) = \sca{\nr{\cdot}^2}{\rho}$ is the second-order moment of $\rho$ and where $\Dual$ corresponds to the \emph{dual value}
\begin{equation*}
    \Dual = \min  \left\{ \int_{\Prob(\Omega)} \sca{\psi^*_\rho}{\rho} \dd \PP(\rho) \mid (\psi_\rho)_\rho \in \L^\infty(\PP; W^{1, \infty}(\Omega)), \quad \int_{\Prob(\Omega)} \psi_\rho(\cdot) \dd \PP(\rho) = \frac{\nr{\cdot}^2}{2} \right\}.
\end{equation*}
In the expression above, $\psi_\rho^*(\cdot) = \sup_{y \in \Omega} \{\sca{\cdot}{y} - \psi_\rho(y)\}$ corresponds to the convex conjugate of $\psi_\rho$ and $\L^\infty(\PP; W^{1, \infty}(\Omega))$ denotes the set of essentially bounded $\PP$-measurable mappings from $\Prob(\Omega)$ to the Sobolev space $W^{1, \infty}(\Omega)$ of bounded Lipschitz continuous functions from $\Omega$ to $\Rsp$.
\end{proposition}
\begin{remark}
Note that in the above minimization problem, $(\psi_\rho)_\rho$ is to be understood as the following mapping, defined $\PP$-almost everywhere:
\begin{equation*}
    (\psi_\rho)_\rho : \left\{
    \begin{array}{ll}
        \Prob(\Omega) &\to W^{1, \infty}(\Omega), \\
        \rho &\mapsto \psi_\rho.
    \end{array}
\right.
\end{equation*}
\end{remark}
\begin{remark}
\label{rk:kantorovich-brenier-potentials}
By Kantorovich duality \cite{villani2008optimal}, for $\PP \in \Prob(\Prob(\Omega))$, the \emph{collection of functions} $(\psi_\rho)_\rho$ solving $\Dual$ gives solutions to the optimal transport problems between $\PP$-a.e. $\rho \in \Prob(\Omega)$ and any barycenter $\mu_\PP \in \arg\min \Primal$: 
\begin{align}
\label{eq:ot-pb-simple}
    \frac{1}{2} \Wass_2^2(\rho, \mu_\PP) &= \frac{1}{2} M_2(\rho) +  \frac{1}{2} M_2(\mu_\PP) - \left( \sca{\psi_\rho^*}{\rho} + \sca{\psi_\rho}{\mu_\PP} \right) \notag \\
    &= \frac{1}{2} M_2(\rho) +  \frac{1}{2} M_2(\mu_\PP) - \left( \min_{\psi \in \Class(\Omega)} \sca{\psi^*}{\rho} + \sca{\psi}{\mu_\PP} \right).
\end{align}
As such, $\psi_\rho = \psi_\rho^{**}$ for $\PP$-a.e. $\rho$, so that this function -- that we call later on a (Kantorovich) potential -- is convex and Lipschitz continuous with Lipschitz constant smaller than $R$. When $\PP(\Prob_{a.c.}(\Omega))>0$ and $\rho \in \spt(\PP) \cap \Prob_{a.c.}(\Omega)$, the convex function $\psi_\rho^*$ is the Brenier potential \cite{Brenier} and its gradients achieves the optimal transport from $\rho$ to the unique barycenter $\mu_\PP$:
\begin{equation*}
    \left( \nabla \psi^*_\rho\right)_\# \rho = \mu_\PP, \quad \text{and} \quad \Wass_2^2(\rho, \mu_\PP) = \nr{\nabla \psi_\rho^* - \id}^2_{\L^2(\rho; \Rsp^d)}.
\end{equation*}
\end{remark}

\subsection{Assumptions}
For any $\PP \in \Prob(\Prob(\Omega))$, the variance functional $F_\PP$ is convex. Stability estimates for the minimizers of $F_\PP$ (which are the Wasserstein barycenters of $\PP$) may thus be obtained from estimates on the \emph{strong convexity} or \emph{curvature} of $F_\PP$. However, without any assumption on $\PP$, the variance functional $F_\PP$ is in general not \emph{strongly-convex} in any sense. In fact, it is easy to construct examples where $F_\PP$ showcases an affine behavior with respect to the linear structure of $\Prob(\Omega)$:
\begin{example}
For any $\PP \in \Prob(\Prob(\Omega))$ of the form $\PP = \sum_i \lambda_i \delta_{\delta_{x_i}}$, one has for any $y, z \in \Omega$ and $t \in [0,1]$ the relation $$ F_\PP((1-t) \delta_{y} + t \delta_{z}) = (1-t)F_\PP(\delta_{y}) + t F_\PP(\delta_{z}). $$
\end{example}
Our main stability applies to measures $\PP\in\Prob(\Prob(\Omega))$ such that the variance functional $F_\PP$ satisfies a strong convexity estimate, also called a \emph{variance inequality} in the language of \cite{sturm-03, pmlr-v125-chewi20a}. Because for any $\PP \in \Prob(\Prob(\Omega))$ we have $$ F_\PP(\cdot) = \frac{1}{2}\int_{\Prob(\Omega)} \Wass_2^2(\rho, \cdot) \dd \PP(\rho),$$
it suffices to choose a $\PP$ whcih gives positive mass to probability measures $\rho \in \Prob(\Omega)$ for which the squared Wasserstein distance $\frac{1}{2} \Wass_2^2(\rho, \cdot)$ satisfies itself a strong convexity estimate. In turn, relying on Kantorovich's dual formulation displayed in \eqref{eq:ot-pb-simple}, this can be obtained from the assumption that the minimized functional in the dual problem presents a form of local strong convexity. We will denote  $\Kant_\rho : \psi \mapsto \sca{\psi^*}{\rho}$ the \emph{Kantorovich functional} associated to $\rho \in \Prob(\Omega)$. This convex functional appears in the minimization problem \eqref{eq:ot-pb-simple}; its gradient formally reads $\nabla \Kant_\rho(\psi) = -(\nabla \psi^*)_\# \rho$. We will make the following assumption:
\begin{assumption}
\label{assump:reg-rho}
There exists constants $\alpha_\PP \in (0, 1]$, $c_\PP, \per_\PP, m_\PP, M_\PP \in (0, +\infty)$ and a measurable set $S_\PP \subset \Prob(\Omega)$ verifying $\PP(S_\PP) = \alpha_\PP$ and such that for all $\rho \in S_\PP$,
\begin{enumerate}
    \item \label{it:ac} $\rho \in \Prob_{a.c.}(\Omega)$,
    \item \label{it:bornes-densite} $m_\PP \leq \rho_{\vert \spt(\rho)} \leq M_\PP$,
    \item \label{it:perimetre} $\spt(\rho)$ has a $\Haus^{d-1}$-rectifiable boundary and $\mathcal{H}^{d-1}(\partial \spt(\rho)) \leq \per_\PP$,
    \item \label{it:ft-convexite} $\forall \psi, \tilde{\psi} \in \Class(\Omega), \quad c_\PP \Var_{\rho}(\tilde{\psi}^* - \psi^*) \leq \Kant_\rho(\tilde{\psi}) - \Kant_\rho(\psi) - \sca{\psi - \tilde{\psi}}{(\nabla \psi^*)_\# \rho}$,
\end{enumerate}
where $\spt(\rho)$ denotes the support of $\rho$, $\partial \spt(\rho)$ denotes the topological boundary of this support and $\mathcal{H}^{d-1}$ denotes the $(d-1)$-dimensional Hausdorff measure.
\end{assumption}

While conditions \eqref{it:ac}, \eqref{it:bornes-densite} and \eqref{it:perimetre} speak for themselves, condition \eqref{it:ft-convexite} might seem ad hoc and difficult to verify. However, conditions under which a measure $\rho \in \Prob_{a.c.}(\Omega)$ verifies the local strong convexity estimate \eqref{it:ft-convexite} of Assumption \ref{assump:reg-rho} are given in \cite{delalande:hal-03164147} as a consequence of the Brascamp-Lieb concentration inequality \cite{brascamp-lieb}. In particular, this estimate holds for an absolutely continuous measure $\rho$, supported on a compact convex set, and whose density is bounded away from zero and infinity. In the appendix, we slightly extend this result to measures supported on a connected union of convex sets, thus showing that the convexity of the support of $\rho$ is not absolutely necessary to get strong convexity of $\Kant_\rho$.
\begin{proposition}
Let $\rho \in \Prob_{a.c.}(\Omega)$ and assume that there exists $m_\rho, M_\rho \in (0, +\infty)$ such that $m_\rho \leq \rho \leq M_\rho$ on $\spt(\rho)$. Assume in addition that $\rho$ satisfies a Poincaré-Wirtinger inequality and that $\spt(\rho)$ is a connected finite union of convex sets. Then there exists $c_\rho > 0$ such that for all $\psi, \tilde{\psi} \in \Class(\Omega),$
\begin{equation*}
    c_\rho \Var_{\rho}(\tilde{\psi}^* - \psi^*) \leq \Kant_\rho(\tilde{\psi}) - \Kant_\rho(\psi) - \sca{\tilde{\psi} - \psi}{\nabla \Kant_\rho(\psi)}.
\end{equation*}
\end{proposition}
We refer to Proposition \ref{prop:strong-convexity-kanto-func-2} of the appendix for a precise statement and a proof.
We conjecture that such a strong convexity estimate actually holds for any absolutely continuous measure
satisfying the Poincaré-Wirtinger inequality, maybe with mild additional assumptions on the density and its support. However, this is not the focus of the present article and we leave this for future work. On a more technical side, we note that the Borel measurability of a set $S_\PP \subset \Prob(\Omega)$ as defined in Assumption \ref{assump:reg-rho} needs to be checked depending on the application. Obviously, measurability holds when the number of marginals is finite ($\PP$ is discrete) and $S_\PP$ is a (finite) subset of these marginals. When the number of marginals is not finite, we note that if $S_\PP$ is made of \emph{all} the measures $\rho \in \Prob(\Omega)$ that satisfy conditions \eqref{it:ac}--\eqref{it:perimetre} and that have convex supports, then the measures of $S_\PP$ satisfy condition \eqref{it:ft-convexite} by \cite{delalande:hal-03164147} and the set $S_\PP$ is closed for the weak topology.

\subsection{Main result and consequences}

\noindent Under Assumption \ref{assump:reg-rho}, we prove that the Wasserstein barycenters depend in a Hölder-continuous way on their marginals:
\begin{theorem}
\label{th:stab-bar}
Let $\PP , \QQ \in \Prob(\Prob(\Omega))$ and assume that $\PP$ satisfies Assumption \ref{assump:reg-rho}. Let $\mu_\PP$ be the barycenter of $\PP$ and $\mu_\QQ$ be a barycenter of $\QQ$. Then
\begin{align*}
    \Wass_2(\mu_\PP, \mu_\QQ) \leq \left( \frac{C_{d, m_\PP, M_\PP, \per_\PP, c_\PP} }{\alpha_\PP} \right)^{1/6} \WWass_1(\PP, \QQ)^{1/6},
\end{align*}
where $ C_{d, m_\PP, M_\PP, \per_\PP, c_\PP} = C_d \frac{M_\PP^3 }{m_\PP} \frac{\per_\PP^2}{c_\PP} R^{5}$ and where $C_d$ is a dimensional constant. It also holds, with the same constant:
\begin{align*}
    \Wass_2(\mu_\PP, \mu_\QQ) \leq \left( \frac{C_{d, m_\PP, M_\PP, \per_\PP, c_\PP} }{\alpha_\PP} \right)^{1/5} \nr{\PP - \QQ}_\TV^{1/5}.
\end{align*}
\end{theorem}

In this result, the Hölder exponents might not be optimal. However, our structure of proof does not leave space to much improvement of these exponents (see Remark \ref{rk:exponent-strong-convexity-W2}). Theorem \ref{th:stab-bar} is essentially a corollary of the fact that whenever a measure $\PP \in \Prob(\Prob(\Omega))$ satisfies Assumption~\ref{assump:reg-rho}, its variance functional $F_\PP$ satisfies a strong convexity estimate (Theorem~\ref{th:strong-convexity-variance-functional}). Note that a strong convexity estimate for the variance functional $F_\PP$ may have an interest beyond the stability of Wasserstein barycenters with respect to their marginals, e.g. to control the \emph{bias} induced by the entropic penalization of the variance functional as introduced in \cite{bigot:hal-01564007} (see Corollary~\ref{cor:bias-penalization}). We defer the detailed proof of Theorem \ref{th:stab-bar} to Section~\ref{sec:strong-convexity-var-functional}. Let us now mention some consequences of Theorem~\ref{th:stab-bar} in applications.

\subsubsection{Statistical estimation of barycenter with a finite number of marginals}
For a probability measure $\rho \in \Prob(\Omega)$ and an i.i.d. sequence $(x_j)_{j=1, \dots, n}$ sampled from $\rho$, it is well-known that the empirical measure $\mathbf{\hat{\rho}}^n = \frac{1}{n} \sum_{j=1}^n \delta_{x_j}$ converges weakly to $\rho$ almost-surely as $n \to \infty$ \cite{varadarajan}. By Theorem 1 of \cite{Fournier2015}, the rate of this convergence can be controlled in expected Wasserstein distance: there exists a constant $C_d$ depending only on $d$ such that 
\begin{equation*}
    \Esp \Wass_2^2(\hat{\rho}^n, \rho) \leq C_d R^2 \left\{
    \begin{array}{ll}
        n^{-1/2} & \mbox{if } d < 4, \\
        n^{-1/2} \log(n) & \mbox{if } d = 4, \\
        n^{-2/d} & \mbox{else,}
    \end{array}
\right.
\end{equation*}
where the expectation is taken with respect to $(x_j)_{j=1, \dots, n} \sim \rho^{\otimes n}$. Theorem \ref{th:stab-bar} together with a double use of Jensen's inequality allows to translate these rates to the statistical estimation of a Wasserstein barycenter with a finite number of marginals:
\begin{corollary}
Let $\PP_m = \sum_{i=1}^m \lambda_i \delta_{\rho_i} \in \Prob(\Prob(\Omega))$ satisfying Assumption \ref{assump:reg-rho}. For all $i \in \{1, \dots, m\}$, denote $\hat{\rho}_i^n = \frac{1}{n} \sum_{j=1}^n \delta_{x_{i,j}}$ an empirical measure built from an i.i.d. sequence $(x_{i,j})_{1 \leq j \leq n}$ sampled from $\rho_i$. Then the barycenters $\mu_{\PP_m}$ of $\PP_m$ and $\mu_{\widehat{\PP}_m^n}$ of $\widehat{\PP}^n_m = \sum_{i=1}^m \lambda_i \delta_{\hat{\rho}^n_i}$ verify
\begin{equation*}
    \Esp \Wass_2^2(\mu_{\widehat{\PP}^n_m}, \mu_{\PP_m}) \lesssim \frac{1}{\alpha_{\PP_m}^{1/3}} \left\{
    \begin{array}{ll}
        n^{-1/12} & \mbox{if } d < 4, \\
        n^{-1/12} \log(n)^{1/6} & \mbox{if } d = 4, \\
        n^{-1/(3d)} & \mbox{else,}
    \end{array}
\right.
\end{equation*}
where $\lesssim$ hides a multiplicative constant that depends on $d, R, m_{\PP_m}, M_{\PP_m}, \per_{\PP_m}$ and $c_{\PP_m}$.
\end{corollary}

\subsubsection{Convergence rate of empirical barycenters in the Wasserstein space}
\label{subsec:cv-rate-emp-bar}
Another statistical question occurs in the setting where the population of marginals $\PP \in \Prob(\Prob(\Omega))$ is only known through samples $(\rho_i)_{1 \leq i \leq m} \sim \PP^{\otimes m}$. Introducing the plug-in estimator $\PP_m = \frac{1}{m} \sum_{i=1}^m \delta_{\rho_i}$, it is natural to wonder how well $\mu_{\PP_m}$ approaches $\mu_\PP$ in terms of $m$. This question, asked in the more general framework of barycenters in Alexandrov spaces, has been the object of recent research \cite{Ahidar-Coutrix2020, LeGouic2022FastCO}. In the Wasserstein space, the authors of \cite{LeGouic2022FastCO} show in particular that $\Esp \Wass_2(\mu_\PP, \mu_{\PP_m})$ converges at the parametric rate $m^{-1/2}$ under the assumption that $\PP$ admits a barycenter $\mu_{\PP}$ that it is such that there exists a bi-Lipschitz optimal transport map between any $\rho \in \spt(\PP)$ and $\mu_\PP$, and that the Lipschitz constants of these maps and their inverses do not differ by a value more than $1$. 
Under similar assumptions, the authors of \cite{pmlr-v125-chewi20a} derive a strong convexity estimate for the variance functional at its minimum which helps them derive rates of convergence of gradient descent algorithms for the (stochastic) estimation of barycenters. Such assumptions however require to have guarantees on the regularity of a barycenter of $\PP$, which can be obtained when restricted to specific families of probability measures (e.g. Gaussian measures), but are difficult to get in general (for instance, barycenters of measures with convex support may not have a convex support \cite{SANTAMBROGIO2016152}, which hampers a straightforward use of Caffarelli's regularity theory). In contrast, our stability result entails that for barycenters $\mu_\PP$ of $\PP$ and $\mu_{\PP_m}$ of $\PP_m$,
\begin{equation*}
    \Esp \Wass_2(\mu_\PP, \mu_{\PP_m}) \lesssim \frac{1}{\alpha_\PP^{1/6}} \Esp \WWass_1(\PP, \PP_m)^{1/6},
\end{equation*}
whenever $\PP$ satisfies Assumption \ref{assump:reg-rho}. This implies that any rate of convergence of $\Esp \WWass_1(\PP, \PP_m)$ with respect to $m$ is readily transferred to $\Esp \Wass_2(\mu_\PP, \mu_{\PP_m})$, up to an exponent. However, $\Prob(\Omega)$ is an infinite dimensional space and there is no general convergence rate for $\Esp \WWass_1(\PP, \PP_m)$. Nonetheless, assuming some structure on the population $\PP$ may help derive convergence bounds. One may use for instance the notion of upper Wasserstein dimension of $\PP$ introduced in \cite{Weed2019SharpAA} (Definition 4), defined from quantities that depend on the covering numbers of (subsets of) the support of $\PP$. Assuming that this dimension is strictly upper bounded by $s>0$, the authors of \cite{Weed2019SharpAA} show that
\begin{equation*}
     \Esp \WWass_1(\PP, \PP_m) \lesssim m^{-1/s},
\end{equation*}
where $\lesssim$ hides a multiplicative constant that depends on $R$ and $s$. We note finally that if we assume that $\PP$ satisfies Assumption \ref{assump:reg-rho} with $\alpha_\PP=1$, our results allow to get the following finite-sample guarantee for the empirical estimation of barycenters in the Wasserstein space:

\begin{theorem} \label{th:empirical-bary-wass-space}
    Let $\PP \in \Prob(\Prob(\Omega))$ satisfying Assumption \ref{assump:reg-rho} with $\alpha_\PP = 1$. For $m \geq 1$, introduce the plug-in estimator $\PP_m = \frac{1}{m} \sum_{i=1}^m \delta_{\rho_i}$ built from an $m$-sample $(\rho_i)_{1 \leq i \leq m} \sim \PP^{\otimes m}$. Then the barycenters $\mu_\PP$ of $\PP$ and $\mu_{\PP_m}$ of $\PP_m$ satisfy
    $$ \Esp \Wass_2(\mu_{\PP_m}, \mu_\PP) \lesssim m^{-1/30},$$
    where $\lesssim$ hides a multiplicative constant that depends on $d, R, m_\PP, M_\PP, \per_\PP$ and $c_\PP$.
\end{theorem}
The main idea to prove this result is to see the minimization of $F_\PP$ through the lens of risk minimization, so that the minimization of $F_{\PP_m}$ corresponds to a problem of empirical risk minimization (ERM) \cite{vapnik-91}. Under the assumptions of Theorem \ref{th:empirical-bary-wass-space}, the empirical risk $F_{\PP_m}$ is ensured to be \emph{strongly-convex} almost surely, which allows to derive stability bounds for the empirical risk minimizer $\mu_{\PP_m}$ with respect to its population counterpart $\mu_\PP$ using classical ideas from the ERM litterature \cite{JMLR:v11:shalev-shwartz10a}. The detailed proof of Theorem \ref{th:empirical-bary-wass-space} is deferred to Section~\ref{sec:ERM}.

\subsubsection{Error induced by a discretization of the marginals}
Let $\rho \in \Prob(\Omega)$ and let $h > 0$ be a discretization parameter. Denoting $(x_i^h)_{1 \leq i \leq N_h}$ an $h$-net of $\Omega$ and $(V_i^h)_{1 \leq i \leq N_h}$ the corresponding Voronoi tessellation of $\Omega$, it is trivial to verify that the discretization $\rho^h = \sum_{i=1}^{N_h} \rho(V_i^h) \delta_{x_i^h}$ verifies $$ \Wass_2(\rho, \rho^h) \leq h.$$ Such a type of discretization, with controlled error bound, may be useful in practice for computational purposes. The stability result of Theorem \ref{th:stab-bar} allows to translate the error bound made when discretizing the marginals to the corresponding barycenter: 
\begin{corollary}
Let $\PP_m = \sum_{i=1}^m \lambda_i \delta_{\rho_i} \in \Prob(\Prob(\Omega))$ satisfying Assumption \ref{assump:reg-rho}. Let $h>0$ and for all $i \in \{1, \dots, m\}$, denote $\rho_i^h = \sum_{j=1}^{N_h} \rho_i(V_j^h) \delta_{x_j^h}$ a discretization of $\rho_i$ built from the $h$-net $(x_j^h)_{1 \leq j \leq N_h}$. Then the barycenters $\mu_{\PP_m}$ of $\PP_m$ and $\mu_{\PP_m^h}$ of $\PP^h_m = \sum_{i=1}^m \lambda_i \delta_{\rho^h_i}$ verify
\begin{equation*}
    \Wass_2(\mu_{\PP^h_m}, \mu_{\PP_m}) \lesssim \frac{1}{\alpha_{\PP_m}^{1/6}} h^{1/6},
\end{equation*}
where $\lesssim$ hides a multiplicative constant that depends on $d, R, m_{\PP_m}, M_{\PP_m}, \per_{\PP_m}$ and $c_{\PP_m}$.
\end{corollary}

\section{Strong convexity of the variance functional}
\label{sec:strong-convexity-var-functional}
Let us recall the definition of the variance functional associated to some $\PP \in \Prob(\Prob(\Omega))$:
\begin{equation*}
    F_\PP : \left\{
    \begin{array}{ll}
        \Prob(\Omega) &\to \Rsp, \\
        \mu &\mapsto \frac{1}{2} \int_{\Prob(\Omega)} \Wass_2^2(\rho, \mu) \dd \PP(\rho).
    \end{array}
\right.
\end{equation*}
Our objective in this section is to get a \emph{strong convexity} estimate for the variance functional $F_\PP$ when $\PP$ satisfies Assumption \ref{assump:reg-rho}. More precisely, we wish to quantify to much extent the graph of the convex functional $F_\PP$ lies above its tangents. For any measure $\mu \in \Prob(\Omega)$, the directions of the tangents of the graph of $F_\PP$ at $\mu$ are given by the subdifferential of $F_\PP$ evaluated at $\mu$ and denoted $\partial F_\PP(\mu)$. This subdifferential may be described using Kantorovich's duality formula \cite{villani2008optimal}, already mentioned in equation \eqref{eq:ot-pb-simple}, that ensures that for any $\rho, \mu \in \Prob(\Omega)$, one has
\begin{equation}
\label{eq:kantorovich-duality-formula}
    \frac{1}{2} \Wass_2^2(\rho, \mu) = \sca{\frac{1}{2}\nr{\cdot}^2}{\rho} +  \sca{\frac{1}{2}\nr{\cdot}^2}{\mu} - \left(\min_{\psi \in \Class(\Omega)} \sca{\psi^*}{\rho} + \sca{\psi}{\mu} \right).
\end{equation}
From this formula, one easily has the following description of the subdifferential of the half-squared Wasserstein distance to a fixed measure $\rho \in \Prob(\Omega)$ (Proposition 7.17 of \cite{santambrogio2015optimal}):
\begin{equation*}
    \partial \left[ \frac{1}{2} \Wass_2^2(\rho, \cdot) \right](\mu) = \left\{ \frac{1}{2}\nr{\cdot}^2 - \psi_{\rho \to \mu} \mid \psi_{\rho \to \mu} \in \arg \min_{\psi \in \Class(\Omega)} \sca{\psi^*}{\rho} + \sca{\psi}{\mu} \right\}.
\end{equation*}
This allows to directly characterize the subdifferential of $F_\PP$ at any $\mu \in \Prob(\Omega)$ as follows:
\begin{equation*}
    \partial F_\PP(\mu) = \left\{ \int_{\Prob(\Omega)} \left(\frac{1}{2}\nr{\cdot}^2 - \psi_{\rho \to \mu} \right) \dd \PP(\rho)  \mid \text{for $\PP$-a.e. $\rho$, } \psi_{\rho \to \mu} \in \arg \min_{\psi \in \Class(\Omega)} \sca{\psi^*}{\rho} + \sca{\psi}{\mu}  \right\}.
\end{equation*}
Thus for any $\mu$ and $\nu$ in $\Prob(\Omega)$, for any \emph{collection} $(\psi_{\rho \to \mu})_\rho \in \L^\infty(\PP; W^{1,\infty}(\Omega))$ of Kantorovich potentials in the transport between $\PP$-almost every $\rho \in \Prob(\Omega)$ and $\mu$, we have by definition of the subdifferential:
\begin{equation*}
    F_\PP(\mu) + \sca{ \int_{\Prob(\Omega)} \left(\frac{1}{2}\nr{\cdot}^2 - \psi_{\rho \to \mu} \right) \dd \PP(\rho) }{\nu - \mu} \leq F_\PP(\nu).
\end{equation*}
Our strong convexity estimate quantifies the gap in this subdifferential inequality under the hypothesis that $\PP$ satisfies Assumption~\ref{assump:reg-rho}:
\begin{theorem}
\label{th:strong-convexity-variance-functional}
    Let $\PP \in \Prob(\Prob(\Omega))$ satisfying Assumption \ref{assump:reg-rho}. Let $\mu, \nu \in \Prob(\Omega)$ and let $(\psi_{\rho \to \mu})_\rho \in \L^\infty(\PP; W^{1,\infty}(\Omega))$ be a \emph{collection} of Kantorovich potentials in the transport between $\PP$-almost every $\rho \in \Prob(\Omega)$ and $\mu$. Then it holds:
    \begin{equation*}
     \alpha_\PP \Wass_2^6(\mu, \nu) \lesssim  F_\PP(\nu) - F_\PP(\mu) - \sca{ \int_{\Prob(\Omega)} \left(\frac{1}{2}\nr{\cdot}^2 - \psi_{\rho \to \mu} \right) \dd \PP(\rho) }{\nu - \mu},
\end{equation*}
    where $\lesssim$ hides on the right-hand side the multiplicative constant 
    $$ C_{d, m_\PP, M_\PP, \per_\PP} = C_d \frac{M_\PP^3}{m_\PP} \frac{\per_\PP^2}{c_\PP} R^4,$$ where $C_d$ is a dimensional constant.
\end{theorem}

\begin{remark}
    This estimate holds without any regularization of the variance functional. As such, it may be used directly to study the stability of \emph{smoothed} notions of Wasserstein barycenters defined from a regularization of the variance functional \cite{bigot:hal-01564007, entropic-barycenters}, yielding stability bounds that may not depend on the regularization parameter(s). Other versions of \emph{smoothed} Wasserstein barycenters have also been obtained from a regularization of the Wasserstein distance itself, such as the celebrated entropic regularization \cite{pmlr-v32-cuturi14}. The stability of these barycenters may also be obtained from similar strong convexity estimates found in the context of entropic optimal transport \cite{pmlr-v151-delalande22a, delalande:tel-03935445}. Finally, the estimate of Theorem \ref{th:strong-convexity-variance-functional} can be used to study the convergence of regularized Wasserstein barycenters towards their non-regularized counterparts, as indicated by the following corollary.
\end{remark}

\begin{corollary}
\label{cor:bias-penalization}
    Let $\PP \in \Prob(\Prob(\Omega))$ satisfying Assumption \ref{assump:reg-rho}. For $\lambda \geq 0$, denote
    $$ \mu_{\PP}^\lambda = \arg\min_{\mu \in \Prob(\Omega)} F_\PP(\mu) + \lambda G(\mu),$$
    where $G : \Prob(\Omega) \to \Rsp$ is either the entropy $G(\mu) = \int_\Omega \mu \log \mu $ or $G(\mu) = \int_\Omega \mu^p$ for some $p \geq 1$. Then for any $\lambda > 0$,
    $$ \Wass_2(\mu^\lambda_\PP, \mu^0_\PP) \lesssim \lambda^{1/6},$$
    where $\lesssim$ hides a multiplicative constant that depends on $d, R, m_{\PP}, M_{\PP}, \per_{\PP}, c_\PP$ and $\alpha_\PP$.
\end{corollary}
\begin{proof}
    Theorem \ref{th:strong-convexity-variance-functional} together with the positiveness of $G$ and the definition of $\mu^\lambda_\PP$ yield
    \begin{equation*}
        \alpha_\PP \Wass_2^6(\mu^\lambda_\PP, \mu^0_\PP) \lesssim F_\PP(\mu^\lambda_\PP) - F_\PP(\mu^0_\PP) \leq F_\PP(\mu^\lambda_\PP) + \lambda G(\mu^\lambda_\PP) - F_\PP(\mu^0_\PP) \leq F_\PP(\mu^0_\PP) + \lambda G(\mu^0_\PP) - F_\PP(\mu^0_\PP).
    \end{equation*}
    The conclusion follows from the boundedness of $G(\mu^0_\PP)$ induced by the maximum principle followed by $\mu_\PP^0 = \mu_\PP \in \Prob_{a.c.}(\Omega)$ when $\PP$ satisfies Assumption \ref{assump:reg-rho} (Proposition 4.7 and Remark 4.8 of \cite{entropic-barycenters}):
    \begin{equation*}
        \nr{\mu_\PP^0}_{\L^\infty} \leq M_\PP / \alpha_\PP^d. \qedhere
    \end{equation*}
\end{proof}

Before proving Theorem \ref{th:strong-convexity-variance-functional}, let us use it to prove the stability estimate of Theorem~\ref{th:stab-bar}.
\begin{proof}[Proof of Theorem \ref{th:stab-bar}]
    Let $(\psi_\rho)_{\rho} = (\psi_{\rho \to \mu_{\PP}})_\rho \in \L^\infty(\PP; W^{1,\infty}(\Omega))$ be a collection of potentials that give a dual solution to the barycenter problem with population $\PP$ (Proposition \ref{prop:dual-formulation}). We have in particular
    $$ \int_{\Prob(\Omega)} \psi_\rho(\cdot) \dd \PP(\rho) = \frac{1}{2}\nr{\cdot}^2.$$
    Applying Theorem \ref{th:strong-convexity-variance-functional} with $\mu = \mu_\PP$, $\nu = \mu_\QQ$ and with the collection of potentials $(\psi_\rho)_{\rho}$, we have the bound
    \begin{align*}
        \alpha_\PP \Wass_2^6(\mu_\PP, \mu_\QQ) &\lesssim F_\PP(\mu_\QQ) - F_\PP(\mu_\PP). 
    \end{align*}
    By definition of $\mu_\QQ$ as a minimizer of $F_\QQ$, we have
    $$ F_\QQ(\mu_\PP) - F_\QQ(\mu_\QQ) \geq 0. $$
    Thus the following bound holds:
    \begin{align}
    \label{eq:stability-bary-1}
        \alpha_\PP \Wass_2^6(\mu_\PP, \mu_\QQ) &\lesssim  F_\PP(\mu_\QQ) - F_\QQ(\mu_\QQ)  + F_\QQ(\mu_\PP) - F_\PP(\mu_\PP) \notag \\
        &= \sca{ \frac{1}{2} \Wass_2^2(\cdot, \mu_\QQ)}{\PP - \QQ} + \sca{ \frac{1}{2} \Wass_2^2(\cdot, \mu_\PP)}{\QQ - \PP} \notag \\
        &= \sca{ \frac{1}{2}(\Wass_2^2(\cdot, \mu_\QQ) - \Wass_2^2(\cdot, \mu_\PP))}{\PP - \QQ}.
    \end{align}
    The mapping $\rho \mapsto \frac{1}{2}(\Wass_2^2(\rho, \mu_\QQ) - \Wass_2^2(\rho, \mu_\PP))$ being $4 R$-Lipschitz continuous with respect to $\Wass_2$, we finally have with the Kantorovich-Rubinstein duality result the bound
    $$ \sca{ \frac{1}{2}(\Wass_2^2(\cdot, \mu_\QQ) - \Wass_2^2(\cdot, \mu_\PP)) }{\PP - \QQ} \leq 4R \WWass_1(\PP, \QQ),$$
    which gives the first estimate of the statement. Now remark that for any $\rho \in \Prob(\Omega)$, the triangle inequality yields
    \begin{equation*}
        \abs{\Wass_2^2(\rho, \mu_\QQ) - \Wass_2^2(\rho, \mu_\PP)} \leq (\Wass_2(\rho, \mu_\QQ) + \Wass_2(\rho, \mu_\PP)) \Wass_2(\mu_\PP, \mu_\QQ) \leq 4R \Wass_2(\mu_\PP, \mu_\QQ).
    \end{equation*}
    Injecting this last bound into \eqref{eq:stability-bary-1} gives the second bound of the statment:
    \begin{equation*}
         \alpha_\PP \Wass_2^6(\mu_\PP, \mu_\QQ) \lesssim  \Wass_2(\mu_\PP, \mu_\QQ) \nr{\PP - \QQ}_{\TV}. \qedhere
    \end{equation*}
\end{proof}

Let us now prove Theorem \ref{th:strong-convexity-variance-functional}. This result simply relies on the fact that for any $\rho \in \Prob(\Omega)$ that belongs to the set $S_\PP$ from Assumption \ref{assump:reg-rho}, the function $\frac{1}{2} \Wass_2^2(\rho, \cdot)$ satisfies a strong convexity estimate given in the following proposition. Theorem \ref{th:strong-convexity-variance-functional} then immediately follows from this proposition after summing over $\rho \sim \PP$. In the following statement, we recall that the \emph{Kantorovich functional} $\Kant_\rho$ associated to a measure $\rho \in \Prob(\Omega)$ corresponds to the map
\begin{equation*}
    \Kant_\rho : \left\{
    \begin{array}{ll}
        \Class(\Omega) &\to \Rsp, \\
        \psi &\mapsto \int_\Omega \psi^* \dd \rho.
    \end{array}
\right.
\end{equation*}

\begin{proposition}
\label{prop:strong-convexity-squared-W2}
Let $\rho \in \Prob_{a.c.}(\Omega)$ be absolutely continuous and such that there exists $m_\rho, M_\rho, \per_\rho, c_\rho \in (0, +\infty)$ verifying
\begin{enumerate}
    \item \label{it:assump2} $m_\rho \leq \rho_{\vert \spt(\rho)} \leq M_\rho$,
    \item \label{it:assump3} $\spt(\rho)$ has a $\Haus^{d-1}$-rectifiable boundary and $\mathcal{H}^{d-1}(\partial \spt(\rho)) \leq \per_\rho$,
    \item \label{it:assump4} $\forall \psi, \tilde{\psi} \in \Class(\Omega), \quad c_\rho \Var_{\rho}(\tilde{\psi}^* - \psi^*) \leq \Kant_\rho(\tilde{\psi}) - \Kant_\rho(\psi) - \sca{\psi - \tilde{\psi}}{(\nabla \psi^*)_\# \rho}$.
\end{enumerate} Then for any $\mu, \nu \in \Prob(\Omega)$ and any Kantorovich potential $\psi_{\rho \to \mu} \in \Class(\Omega)$ in the optimal transport between $\rho$ and $\mu$, one has
\begin{equation*}
        \forall \mu, \nu, \quad  \Wass_2^6(\mu, \nu) \lesssim \frac{1}{2}\Wass_2^2(\nu, \rho) - \frac{1}{2}\Wass_2^2(\mu, \rho) - \sca{\frac{1}{2}\nr{\cdot}^2 - \psi_{\rho \to \mu}}{\nu - \mu},
    \end{equation*}
    where $\lesssim$ hides on the right-hand side the multiplicative constant 
    $$ C_{d, m_\rho, M_\rho, \per_\rho} = C_d \frac{M_\rho^3}{m_\rho} \frac{\per_\rho^2}{c_\rho} R^4,$$ where $C_d$ is a dimensional constant.
\end{proposition}

\begin{remark}[\emph{Linear} convexity vs. \emph{displacement} convexity]
    We emphasize on the fact that Proposition~\ref{prop:strong-convexity-squared-W2} gives a strong convexity estimate for $\frac{1}{2}\Wass_2^2(\rho, \cdot)$ with respect to the \emph{linear} structure on $\Prob(\Omega)$, and not with respect to the metric structure of $(\Prob(\Omega), \Wass_2)$. Convexity of a functional with respect to this structure is referred to the notion of \emph{displacement} convexity. Strong convexity of $\frac{1}{2}\Wass_2^2(\rho, \cdot)$ with respect to the metric structure of $(\Prob(\Omega), \Wass_2)$ is trivial to get in dimension $d=1$ because of the Hilbertian nature of $\Wass_2$ in this context (see Section~\ref{sec:stab-dim-1}). However, this is limited to the unidimensional setting and $\frac{1}{2}\Wass_2^2(\rho, \cdot)$ is notoriously \emph{not} displacement convex in dimension $d \geq 2$ (see for instance Section 7.3.3 of \cite{santambrogio2015optimal}).
\end{remark}

\begin{remark}[Exponent]
\label{rk:exponent-strong-convexity-W2}
    We note that the value $6$ of the exponent on the left-hand side term of the estimate of Proposition~\ref{prop:strong-convexity-squared-W2} might not be optimal. However, $4$ should be a lower-bound on the value of this exponent. This may be seen from the following example: in dimension $d=1$ and for $\eps > 0$, set $\mu^\eps = (\frac{1}{2} - \frac{\eps}{2})(\delta_{-1} + \delta_{1}) +  \eps \delta_0$. Then we have $\Wass_2^2(\mu^0, \mu^\eps) = \eps$. For $\rho = \lambda_{\vert [-\frac{1}{2}, \frac{1}{2}]}$, the following computation holds:
\begin{gather*}
    \frac{1}{2}\Wass_2^2(\mu^\eps, \rho) - \frac{1}{2}\Wass_2^2(\mu^0, \rho) = \int_0^{\eps/2} (\abs{0-x}^2 - \abs{1-x}^2) \dd x = \frac{\eps^2}{4} - \frac{\eps}{2}.
\end{gather*}
Finally, one can choose a Kantorovich potential in the transport from $\rho$ to $\mu^0$ to be $\psi_{\rho \to \mu^0} = \iota_{[-1,1]}$ (i.e. valued $0$ on $[-1,1]$ and $+\infty$ outside this segment), so that
$$ \sca{\frac{\nr{\cdot}^2}{2} - \psi_{\rho \to \mu^0}}{\mu^\eps - \mu^0} = \sca{\frac{\nr{\cdot}^2}{2} }{\mu^\eps - \mu^0} = - \frac{\eps}{2}.$$
Hence:
$$ \frac{1}{2}\Wass_2^2(\mu^\eps, \rho) - \frac{1}{2}\Wass_2^2(\mu^0, \rho) - \sca{\frac{\nr{\cdot}^2}{2} - \psi_{\rho \to \mu^0}}{\mu^\eps - \mu^0} = \frac{\eps^2}{4} = \frac{1}{4} \Wass_2^4(\mu^0, \mu^\eps). $$
\end{remark}

\begin{proof}[Proof of Proposition \ref{prop:strong-convexity-squared-W2}]
    Let $\psi_{\rho \to \nu} \in \Class(\Omega)$ be a Kantorovich potential in the optimal transport between $\rho$ and $\nu$. Then the conjugates $\psi_{\rho \to \mu}^*, \psi_{\rho \to \nu}^*$ are both convex Brenier potentials \cite{Brenier} in the optimal transport between the absolutely continuous source $\rho$ and the targets $\mu, \nu$, in the sense that:
    \begin{equation*}
        (\nabla \psi_{\rho \to \mu}^*)_\# \rho = \mu \quad \text{and} \quad (\nabla \psi_{\rho \to \nu}^*)_\# \rho = \nu.
    \end{equation*}
    Therefore, the coupling $(\nabla \psi_{\rho \to \mu}^*, \nabla \psi_{\rho \to \nu}^*)_\# \rho$ is an admissible transport plan between $\mu$ and $\nu$ and as such:
    \begin{equation}
    \label{eq:strong-cvxity-W2-1}
        \Wass_2^2(\mu, \nu) \leq \nr{ \nabla \psi_{\rho \to \mu}^* - \nabla \psi_{\rho \to \nu}^*}^2_{\L^2(\rho; \Rsp^d)}.
    \end{equation}
    We now quote a Gagliardo-Nirenberg type inequality, extracted from Proposition~4.1 in \cite{delalande:hal-03164147}, that ensures that for any compact domain $K$ of $\Rsp^d$ with $\mathcal{H}^{d-1}$-rectifiable boundary and $u, v: K \to \Rsp$ two Lipschitz functions on $K$ that are convex on any segment included in $K$, there exists a constant $C_{d}$ depending only on $d$ such that
\begin{equation*}
    \nr{\nabla u - \nabla v}_{\L^2(K)}^6 \leq C_{d} \mathcal{H}^{d-1}(\partial K)^{2} ( \nr{\nabla u}_{\L^\infty(K)} + \nr{\nabla v}_{\L^\infty(K)} )^{4}  \nr{u - v}^{2}_{\L^2(K)}.
\end{equation*}
    We note from \cite{delalande:hal-03164147} that the exponents in this inequality are optimal.  Using that the Brenier potentials $\psi_{\rho \to \mu}^*, \psi_{\rho \to \nu}^*$ are both convex and $R$-Lipschitz continuous and leveraging assumptions \eqref{it:assump2} and \eqref{it:assump3} made on $\rho$, we can apply this inequality to get that for any constant $c \in \Rsp$:
    \begin{equation*}
    \frac{1}{M_\rho^3} \nr{\nabla \psi_{\rho \to \mu}^* - \nabla \psi_{\rho \to \nu}^*}_{\L^2(\rho; \Rsp^d)}^6 \leq C_{d} (\per_\rho)^{2} ( 2 R )^{4} \frac{1}{m_\rho} \nr{\psi_{\rho \to \mu}^* - \psi_{\rho \to \nu}^* - c}^{2}_{\L^2(\rho)}.
\end{equation*}
Minimizing over $c \in \Rsp$ in the last inequality yields:
\begin{equation}
\label{eq:strong-cvxity-W2-2}
    \nr{\nabla \psi_{\rho \to \mu}^* - \nabla \psi_{\rho \to \nu}^*}_{\L^2(\rho; \Rsp^d)}^6 \lesssim \Var_\rho( \psi_{\rho \to \mu}^* - \psi_{\rho \to \nu}^*).
\end{equation}
But assumption \eqref{it:assump4} on $\rho$ ensures:
\begin{equation}
\label{eq:strong-cvxity-W2-3}
    c_\rho \Var_\rho( \psi_{\rho \to \mu}^* - \psi_{\rho \to \nu}^*) \leq \Kant_\rho(\psi_{\rho \to \mu}) - \Kant_\rho(\psi_{\rho \to \nu}) + \sca{\psi_{\rho \to \mu} - \psi_{\rho \to \nu}}{\nu}.
\end{equation}
Finally, notice that by Kantorovich's duality formula \eqref{eq:kantorovich-duality-formula} and by definition of $\psi_{\rho \to \mu}, \psi_{\rho \to \mu}$ as Kantorovich potentials, one has:
\begin{gather*}
    \frac{1}{2} \Wass_2^2(\rho, \mu) = \sca{\frac{1}{2}\nr{\cdot}^2}{\rho} +  \sca{\frac{1}{2}\nr{\cdot}^2}{\mu} - \Kant_\rho(\psi_{\rho \to \mu}) - \sca{\psi_{\rho \to \mu}}{\mu}, \\
    \frac{1}{2} \Wass_2^2(\rho, \nu) = \sca{\frac{1}{2}\nr{\cdot}^2}{\rho} +  \sca{\frac{1}{2}\nr{\cdot}^2}{\nu} - \Kant_\rho(\psi_{\rho \to \nu}) - \sca{\psi_{\rho \to \nu}}{\nu}.
\end{gather*}
This yields:
\begin{equation}
\label{eq:strong-cvxity-W2-4}
    \Kant_\rho(\psi_{\rho \to \mu}) - \Kant_\rho(\psi_{\rho \to \nu}) + \sca{\psi_{\rho \to \mu} - \psi_{\rho \to \nu}}{\nu} = \frac{1}{2}\Wass_2^2(\nu, \rho) - \frac{1}{2}\Wass_2^2(\mu, \rho) - \sca{\frac{1}{2}\nr{\cdot}^2 - \psi_{\rho \to \mu}}{\nu - \mu}.
\end{equation}
The conclusion follows after combining \eqref{eq:strong-cvxity-W2-1}, \eqref{eq:strong-cvxity-W2-2}, \eqref{eq:strong-cvxity-W2-3} and \eqref{eq:strong-cvxity-W2-4} together.
\end{proof}

\section{Convergence of empirical barycenters in the Wasserstein space}
\label{sec:ERM}

This section is devoted to the proof of Theorem~\ref{th:empirical-bary-wass-space}. This proof relies on a classical symmetrization technique used in the study of empirical processes \cite{van1996weak}, already employed in the context of strongly-convex empirical risk minimization (see e.g. the proof of Theorem 2 in \cite{JMLR:v11:shalev-shwartz10a}).

\begin{proof}[Proof of Theorem~\ref{th:empirical-bary-wass-space}]
    Applying the strong convexity estimate of Theorem \ref{th:strong-convexity-variance-functional} to $F_\PP$ at the minimizer $\mu = \mu_\PP$ and with $\nu = \mu_{\PP_m}$, we have the bound
    \begin{align}
    \label{eq:control-conv-empirical-bar-1}
        \Wass_2^6(\mu_\PP, \mu_{\PP_m}) &\lesssim F_\PP(\mu_{\PP_m}) - F_\PP(\mu_\PP) \notag \\
        &\leq F_\PP(\mu_{\PP_m}) - F_{\PP_m}(\mu_{\PP_m}) + F_{\PP_m}(\mu_\PP) - F_\PP(\mu_\PP)
    \end{align}
    We now proceed to the control of the expectation with respect to $(\rho_i)_{1 \leq i \leq m} \sim \PP^{\otimes m}$ of the above two differences. \\ \\
    \textbf{Control of $\Esp( F_\PP(\mu_{\PP_m}) - F_{\PP_m}(\mu_{\PP_m}) )$.} Notice that 
    \begin{equation*}
        F_\PP(\mu_{\PP_m}) - F_{\PP_m}(\mu_{\PP_m}) = \frac{1}{2} \int_{\Prob(\Omega)} \Wass_2^2(\rho, \mu_{\PP_m}) \dd \PP(\rho) - \frac{1}{2m} \sum_{i=1}^m \Wass_2^2(\rho_i, \mu_{\PP_m}).
    \end{equation*}
    In order to control the expectation of this difference, we introduce another i.i.d. $m$-sample of $\PP$:  $(\rho_i')_{1 \leq i \leq m} \sim \PP^{\otimes m}$. One can then notice that
    \begin{align}
    \label{eq:control-conv-empirical-bar-2}
         \Esp \frac{1}{2} \int_{\Prob(\Omega)} \Wass_2^2(\rho, \mu_{\PP_m}) \dd \PP(\rho) &= \Esp_{ (\rho_i)_{i} \sim \PP^{\otimes m} } \Esp_{\rho \sim \PP} \frac{1}{2}  \Wass_2^2(\rho, \mu_{\PP_m}) \notag \\
         &= \Esp_{ (\rho_i)_{i} \sim \PP^{\otimes mN} } \Esp_{ (\rho_i' )_{i} \sim \PP^{\otimes m} }  \frac{1}{2m} \sum_{i=1}^m \Wass_2^2(\rho_i', \mu_{\PP_m}) \notag \\
         &= \Esp \frac{1}{2m} \sum_{i=1}^m \Wass_2^2(\rho_i', \mu_{\PP_m}),
    \end{align}
    where the last expectation is against all the random variables $(\rho_i)_{1 \leq i \leq m} \sim \PP^{\otimes m}$ and $(\rho_i')_{1 \leq i \leq m} \sim \PP^{\otimes m}$.
    Now for any $i \in \{1, \dots, m\}$, introduce the empirical measure
    $$ \PP_m^{(i)} = \frac{1}{m} \sum_{j=1, j\neq i}^m \delta_{\rho_j} + \frac{1}{m} \delta_{\rho'_i}.$$
    Then for any $i\in \{1, \dots, m\}$, taking again the expectation against all the random variables $(\rho_i)_{1 \leq i \leq m} \sim \PP^{\otimes m}$ and $(\rho_i')_{1 \leq i \leq m} \sim \PP^{\otimes m}$ one has
    $$ \Esp \Wass_2^2(\rho_i, \mu_{\PP_m}) = \Esp \Wass_2^2(\rho_i', \mu_{\PP_m^{(i)}}).$$
    This ensures the equality
    \begin{equation}
        \label{eq:control-conv-empirical-bar-3}
        \Esp \frac{1}{2m} \sum_{i=1}^m \Wass_2^2(\rho_i, \mu_{\PP_m}) = \Esp \frac{1}{2m} \sum_{i=1}^m \Wass_2^2(\rho_i', \mu_{\PP_m^{(i)}}).
    \end{equation}
    From equations \eqref{eq:control-conv-empirical-bar-2} and \eqref{eq:control-conv-empirical-bar-3}, the expectation of the first difference appearing in \eqref{eq:control-conv-empirical-bar-1} reads:
    \begin{align*}
        \Esp (F_\PP(\mu_{\PP_m}) - F_{\PP_m}(\mu_{\PP_m})) &=\frac{1}{2m} \sum_{i=1}^m \Esp( \Wass_2^2(\rho_i', \mu_{\PP_m}) - \Wass_2^2(\rho_i', \mu_{\PP_m^{(i)}}) ).
    \end{align*}
    Using that $\Omega = B(0, R)$ is bounded and the triangle inequality, we have the bound
    \begin{align*}
        \frac{1}{2m} \sum_{i=1}^m \Esp( \Wass_2^2(\rho_i', \mu_{\PP_m}) - \Wass_2^2(\rho_i', \mu_{\PP_m^{(i)}}) ) &\leq \frac{(2R + 2R)}{2m} \sum_{i=1}^m \Esp \Wass_2 ( \mu_{\PP_m}, \mu_{\PP_m^{(i)}} ) \\
        &= 2 R \Esp \Wass_2 ( \mu_{\PP_m}, \mu_{\PP_m^{(i)}} ).
    \end{align*}
    Using that $\PP$ satisfies Assumption~\ref{assump:reg-rho} with $\alpha_\PP = 1$, we have that $\PP_m$ (or $\PP_m^{(i)}$) almost surely satisfies Assumption~\ref{assump:reg-rho} with $\alpha_{\PP_m} = 1$ (and with the same other constants as $\PP$ in this assumption). Thus Theorem~\ref{th:stab-bar} ensures almost surely the bound:
    \begin{equation*}
        \Wass_2 ( \mu_{\PP_m}, \mu_{\PP_m^{(i)}} ) \lesssim \nr{ \PP_m - \PP_m^{(i)} }_\TV^{1/5} \lesssim \frac{1}{m^{1/5}}.
    \end{equation*}
    We thus have the following bound on the expectation of the first difference appearing in \eqref{eq:control-conv-empirical-bar-1}:
    \begin{equation}
        \label{eq:control-conv-empirical-bar-4}
        \Esp (F_\PP(\mu_{\PP_m}) - F_{\PP_m}(\mu_{\PP_m})) \lesssim \frac{1}{m^{1/5}}.
    \end{equation}
    \textbf{Control of $\Esp( F_{\PP_m}(\mu_\PP) - F_\PP(\mu_\PP) )$.} Notice that
    \begin{equation*}
        F_{\PP_m}(\mu_\PP) - F_\PP(\mu_\PP) = \frac{1}{2m} \sum_{i=1}^m \Wass_2^2(\rho_i, \mu_{\PP}) - \frac{1}{2} \int_{\Prob(\Omega)} \Wass_2^2(\rho, \mu_{\PP}) \dd \PP(\rho)
    \end{equation*}
     Bounding the expectation of this second difference term is much more straightforward. For any $i \in \{1, \dots, m\}$, denote $X_i = \frac{1}{2} \Wass_2^2(\rho_i, \mu_\PP)$ the scalar random variable built from the random sample $\rho_i \sim \PP$. Denote $\Esp X$ the expectation of this random variable (independent of $i$). Using this notation we can write the expectation of the second difference term of \eqref{eq:control-conv-empirical-bar-1} as follows:
    \begin{align*}
        \frac{1}{2m} \sum_{i=1}^m \Wass_2^2(\rho_i, \mu_{\PP}) - \frac{1}{2} \int_{\Prob(\Omega)} \Wass_2^2(\rho, \mu_{\PP}) \dd \PP(\rho) = \frac{1}{m}\sum_{i=1}^m X_i - \Esp X.
    \end{align*}
    Using Jensen's inequality, we thus have:
    \begin{align}
    \label{eq:control-conv-empirical-bar-5}
        \Esp( F_{\PP_m}(\mu_\PP) - F_\PP(\mu_\PP) ) = \Esp \left(\frac{1}{m}\sum_{i=1}^m X_i - \Esp X\right)\leq \left( \Var \left(\frac{1}{m} \sum_{i=1}^m X_i\right) \right)^{1/2} \lesssim \frac{1}{m^{1/2}}.
    \end{align}
    \textbf{Conclusion.}
    Injecting the bounds \eqref{eq:control-conv-empirical-bar-4} and \eqref{eq:control-conv-empirical-bar-5} in the expectation of \eqref{eq:control-conv-empirical-bar-1} thus yields
    \begin{align*}
        \Esp \Wass_2^6(\mu_\PP, \mu_{\PP_m}) \lesssim  \frac{1}{m^{1/5}} + \frac{1}{m^{1/2}} \lesssim \frac{1}{m^{1/5}}.
    \end{align*}
    Jensen's inequality used in the above bound finally yields the statement.
\end{proof}

\subsection*{Acknowledgement} 
The authors acknowledge the support of the Lagrange Mathematics and Computing Research Center and of the ANR (MAGA, ANR-16-CE40-0014). We thank Blanche Buet for interesting discussions related to this work.

\medskip

\appendix

\section{Dual formulation for the Wasserstein barycenter problem}
\label{sec:dual-formulation}

\begin{proof}[Proof of Proposition \ref{prop:dual-formulation}]
Instead of showing directly the formulation of Proposition \ref{prop:dual-formulation}, we will rather show
\begin{align*}
    \min_{\mu \in \Prob(\Omega)} F_\PP(\mu) = \max \Bigg\{ \int_{\Prob(\Omega)} \sca{\phi^c_\rho}{\rho} \dd \PP(\rho) \mid (\phi_\rho)_\rho \in \L^\infty(\PP; W^{1, \infty}(\Omega)), \quad \int_{\Prob(\Omega)} \phi_\rho(\cdot) \dd \PP(\rho) = 0  \Bigg\},
\end{align*} 
where for any $\rho \in \Prob(\Omega)$, $\phi_\rho^c$ denotes the following $c$-transform of $\phi_\rho$: $\phi_\rho^c(\cdot) = \inf_{y \in \Omega} \frac{1}{2} \nr{\cdot-y}^2 - \phi_\rho(y)$. Such a formulation entails the result of Proposition \ref{prop:dual-formulation} by the change of variable $(\psi_\rho)_\rho = \frac{\nr{\cdot}^2}{2} - (\phi_\rho)_\rho \in \L^\infty(\PP; W^{1, \infty}(\Omega))$.\\
\textbf{Duality.} Let's first show that the value of $\min_{\mu \in \Prob(\Omega)} F_\PP(\mu)$ is equal to the value of the following supremum
\begin{align*}
    \Dual' := \sup \Bigg\{& \int_{\Prob(\Omega)} \sca{\phi^c_\rho}{\rho} \dd \PP(\rho) \mid (\phi_\rho)_\rho \in \L^1(\PP; \Class(\Omega)), \quad \int_{\Prob(\Omega)} \phi_\rho(\cdot) \dd \PP(\rho) = 0  \Bigg\},
\end{align*}
where $\L^1(\PP; \Class(\Omega))$ denotes the set of $\PP$-measurable and Bochner integrable mappings from $\Prob(\Omega)$ to the space $(\Class(\Omega), \nr{\cdot}_\infty)$ of continuous function from $\Omega$ to $\Rsp$ equipped with the supremum norm. Introduce the functional $H : \Class(\Omega) \to \Rsp$ defined for all $\varphi \in \Class(\Omega)$ by
\begin{align*}
   H(\varphi) = \inf  \Bigg\{& -\int_{\Prob(\Omega)} \sca{\phi^c_\rho}{\rho} \dd \PP(\rho) \mid (\phi_\rho)_\rho \in \L^1(\PP; \Class(\Omega)), \quad \int_{\Prob(\Omega)} \phi_\rho(\cdot) \dd \PP(\rho) = \varphi(\cdot)  \Bigg\}.
\end{align*} 
Notice then that $\Dual' = -H(0)$. On the other hand, notice that $H$ has the following convex conjugate: for $\mu \in \Prob(\Omega)$, 
\begin{alignat*}{2}
    H^*(\mu) &= \sup \left\{ \sca{\varphi}{\mu} - H(\varphi) \mid \varphi \in \Class(\Omega) \right\} \\
    &= \sup \Bigg\{ \sca{\varphi}{\mu} + \int_{\Prob(\Omega)} \sca{\phi^c_\rho}{\rho} \dd \PP(\rho) \mid \varphi \in \Class(\Omega), (\phi_\rho)_\rho \in \L^1(\PP; \Class(\Omega)), \int_{\Prob(\Omega)} \phi_\rho(\cdot) \dd \PP(\rho) = \varphi(\cdot) \Bigg\} \\
    &= \sup \left\{ \int_{\Prob(\Omega)} \sca{ \phi_\rho }{\mu} \dd \PP(\rho) + \int_{\Prob(\Omega)} \sca{\phi^c_\rho}{\rho} \dd \PP(\rho), \quad (\phi_\rho)_\rho \in \L^1(\PP; \Class(\Omega)) \right\} \\
    &= \int_{\Prob(\Omega)} \left( \sup_{\phi_\rho \in \Class(\Omega)} \sca{ \phi_\rho }{\mu} + \sca{\phi^c_\rho}{\rho} \right) \dd \PP(\rho) \\
    &=  \int_{\Prob(\Omega)} \frac{1}{2} \Wass_2^2(\mu, \rho) \dd \PP(\rho),
\end{alignat*}
where we used the Kantorovich duality formula (see for instance \cite{villani2008optimal}) to get to the last line. We thus have $$\min_{\mu \in \Prob(\Omega)} F_\PP(\mu) = \inf_{\mu \in \Prob(\Omega)} H^*(\mu) = - H^{**}(0).$$
Therefore, showing that $\Dual' = \min_{\mu \in \Prob(\Omega)} F_\PP(\mu)$ corresponds to show that $H(0) = H^{**}(0)$. Since $H$ is convex (by concavity of the $c$-transform operation), this will follow from the continuity of $H$ at $0$ for the supremum-norm over $\Class(\Omega)$ (Proposition 4.1 of \cite{convex_analysis}). For this, we can first notice that $H$ never takes the value $-\infty$: for any $\varphi \in \Class(\Omega)$ and $(\phi_\rho)_\rho \in \L^1(\PP; \Class(\Omega))$ such that $\int_{\Prob(\Omega)} \phi_\rho(\cdot) \dd \PP(\rho) = \varphi(\cdot)$, one has
\begin{equation*}
    \forall \rho \in \Prob(\Omega),\quad -\phi_\rho^c(x) = \sup_{y \in \Rsp^d} \phi_\rho(y) - \frac{1}{2}\nr{x - y}^2 \geq \phi_\rho(0) - \frac{1}{2} \nr{x}^2.
\end{equation*}
If follows that 
\begin{equation*}
    H(\varphi) \geq \varphi(0) - \int_{\Prob(\Omega)} \frac{M_2(\rho)}{2} \dd \PP(\rho) > -\infty.
\end{equation*}
On the other hand, notice that $H$ is bounded from above in a neighborhood of $0$ in $\Class(\Omega)$: for any $\varphi \in \Class(\Omega)$ such that $\nr{\varphi}_\infty \leq 1$, one has $-\varphi^c(x) \leq 1$ for any $x \in \Rsp^d$ so that
\begin{equation*}
    H(\varphi) \leq - \int_{\Prob(\Omega)} \sca{(\varphi)^c}{\rho} \dd \PP(\rho) \leq 1.
\end{equation*}
A standard convex analysis result (Proposition 2.5 in \cite{convex_analysis}) then ensures that $H$ is continuous at $0$, so that $H(0) = H^{**}(0)$ and $\Dual' = \min_{\mu \in \Prob(\Omega)} F_\PP(\mu)$.\\
\textbf{Restriction to $\L^\infty(\PP; W^{1, \infty}(\Omega))$.}
We show here that we can run the supremum $\Dual'$ only over $\L^\infty(\PP; W^{1, \infty}(\Omega))$ instead of $\L^1(\PP; \Class(\Omega))$, that is
\begin{align*}
    \Dual' = \sup \Bigg\{ \int_{\Prob(\Omega)} \sca{\phi^c_\rho}{\rho} \dd \PP(\rho) \mid (\phi_\rho)_\rho \in \L^\infty(\PP; W^{1, \infty}(\Omega)), \quad \int_{\Prob(\Omega)} \phi_\rho(\cdot) \dd \PP(\rho) = 0  \Bigg\}.
\end{align*}
Let $(\phi_\rho)_\rho \in \L^1(\PP; \Class(\Omega))$ be an admissible solution to $\Dual'$, i.e. $(\phi_\rho)_\rho$ satisfies
\begin{equation}
\label{eq:admissibility-condition}
    \int_{\Prob(\Omega)} \phi_\rho(\cdot) \dd \PP(\rho) = 0.
\end{equation}
Then we can build from $(\phi_\rho)_\rho$ another admissible solution $(\tilde{\phi}_\rho)_\rho$ that belongs to $\L^\infty(\PP; W^{1, \infty}(\Omega))$ and that performs better at $\Dual'$, i.e. that verifies
\begin{equation}
\label{eq:better-candidate}
    \int_{\Prob(\Omega)} \sca{\tilde{\phi}^c_\rho}{\rho} \dd \PP(\rho) \geq \int_{\Prob(\Omega)} \sca{\phi^c_\rho}{\rho} \dd \PP(\rho).
\end{equation}
Indeed, introduce $(\hat{\phi}_\rho)_\rho := (\phi^{cc}_\rho)_\rho$. Then for all $\rho \in \Prob(\Omega)$, $\hat{\phi}_\rho = \phi^{cc}_\rho$ is obviously $2R$-Lipschitz (as a $c$-transform) and satisfies $\hat{\phi}_\rho^c = \phi_\rho^c$ and $\hat{\phi}_\rho \geq \phi_\rho$ (as a double $c$-transform). Using then \eqref{eq:admissibility-condition}, one has that
\begin{equation*}
     \alpha(\cdot) := \int_{\Prob(\Omega)} \hat{\phi}_\rho(\cdot) \dd \PP(\rho) \geq 0,
\end{equation*}
where $\alpha$ is also $2R$-Lipschitz. Now denoting $\tilde{\phi}_\rho = \hat{\phi_\rho} - \alpha$ for all $\rho \in \Prob(\Omega)$, the mapping $(\tilde{\phi}_\rho)_\rho \in \L^1(\PP; \Class(\Omega))$ is admissible to $\Dual'$ by construction and satisfies $\tilde{\phi}_\rho \leq \hat{\phi}_\rho$ for all $\rho \in \Prob(\Omega)$, so that $\tilde{\phi}^c_\rho \geq \hat{\phi}^c_\rho = \phi^c_\rho$ (using that the $c$-transform is order-reversing). For each $\rho \in \Prob(\Omega)$, up to subtracting $\tilde{\phi}_\rho(0)$ to $\tilde{\phi}_\rho$ (this operation leaves $(\tilde{\phi}_\rho)_\rho$ admissible to $\Dual'$ and does not change its value), one can assume that $\tilde{\phi}_\rho(0) = 0$. Noticing that $\tilde{\phi}_\rho$ is $4R$-Lipschitz by construction, we have the bound $\nr{\tilde{\phi}_\rho}_{W^{1, \infty}(\Omega)} \leq 4R(1+R)$. We thus have built an admissible $(\tilde{\phi}_\rho)_\rho \in \L^\infty(\PP; W^{1,\infty}(\Omega))$ from an admissible $(\phi_\rho)_\rho \in \L^1(\PP; \Class(\Omega))$ that satisfies \eqref{eq:better-candidate}, which shows that we can run the supremum $\Dual'$ only over $\L^\infty(\PP; W^{1, \infty}(\Omega))$ instead of $\L^1(\PP; \Class(\Omega))$\\
\textbf{Existence of a maximizer.} 
There now remains to show that the supremum in $\Dual'$ can be replaced by a maximum. Let $\left((\phi_\rho^n)_\rho\right)_{n\geq0}$ be a maximizing sequence to $\Dual'$, and assume from what precedes that this sequence belongs to $\L^\infty(\PP; W^{1, \infty}(\Omega))$ and satisfies for all $n\geq 0$ and $\rho \in \Prob(\Omega)$, $\nr{\phi^n_\rho}_{W^{1, \infty}(\Omega)} \leq 4R(1+R)$. Further assume that this sequence verifies for all $n \geq 1$,
\begin{equation}
\label{eq:maximizing-values}
    \int_{\Prob(\Omega)} \sca{(\phi^n_\rho)^c}{\rho} \dd \PP(\rho) \geq \Dual' - \frac{1}{n}.
\end{equation}
For any $n \geq 0$, the mapping $(\rho, x) \mapsto \phi^n_\rho(x)$ is bounded in $\L^2(\PP \otimes \lambda)$ where $\lambda$ denotes the Lebesgue measure over $\Omega$. Therefore, by Banach-Alaoglu theorem, the sequence $\left((\phi_\rho^n)_\rho\right)_{n\geq0}$ (seen as a sequence in $\L^2(\PP \otimes \lambda)$) admits a weakly converging subsequence in $\L^2(\PP \otimes \lambda)$, that we do not relabel and for which we denote $(\phi^\infty_\rho)_\rho$ the weak limit in $\L^2(\PP \otimes \lambda)$. Using now Mazur's lemma, we know that there exists a sequence of integers $(N_n)_{n \geq 0}$ and coefficients $((\lambda_{n,k})_{n \leq k \leq N_n})_{n \geq 0} \geq 0$ satisfying for all $n \geq 0$, $\sum_{k=n}^{N_n} \lambda_{n, k} = 1$ such that the sequence $\left((\bar{\phi}_\rho^n)_\rho\right)_{n\geq0}$ defined for all $n \geq 0$ and $\rho \in \Prob(\Omega)$ by $\bar{\phi}_\rho^n := \sum_{k=n}^{N_n} \lambda _{n,k} \phi_\rho^k$ converges strongly to $(\phi^\infty_\rho)_\rho$ in $\L^2(\PP \otimes \lambda)$. By concavity of the $c$-transform operation and equation \eqref{eq:maximizing-values}, we then have the bound 
\begin{align}
\label{eq:bar-phi-max-seq}
    \int_{\Prob(\Omega)} \sca{(\bar{\phi}^n_\rho)^c}{\rho} \dd \PP(\rho) &\geq \sum_{k = n}^{N_n} \lambda_{n,k} \int_{\Prob(\Omega)} \sca{(\phi^k_\rho)^c}{\rho} \dd \PP(\rho) \notag \\
    &\geq \sum_{k = n}^{N_n} \lambda_{n,k} \left( \Dual' - \frac{1}{k} \right) \notag \\
    &\geq \Dual' - \frac{1}{n}.
\end{align}
The sequence $\left((\bar{\phi}_\rho^n)_\rho\right)_{n\geq0}$ is therefore also a maximizing sequence of $\Dual'$ and it also satisfies for any $n \geq 0$ and $\rho \in \Prob(\Omega)$ the bound 
\begin{equation}
    \label{eq:regularity-sequence-cvx-comb}
    \nr{\bar{\phi}^n_\rho}_{W^{1, \infty}(\Omega)} \leq 4R(1+R).
\end{equation}
Since the sequence $\left((\bar{\phi}_\rho^n)_\rho\right)_{n\geq0}$ strongly converges to $(\phi^\infty_\rho)_\rho$ in $\L^2(\PP \otimes \lambda)$, one can extract a subsequence (that we do not relabel) such that for $\PP$-almost-every $\rho \in \Prob(\Omega)$, the sequence $(\bar{\phi}^n_\rho)_{n \geq 0}$ converges to $\phi^\infty_\rho$ in $\L^2(\lambda)$. Using \eqref{eq:regularity-sequence-cvx-comb} and Arzelà-Ascoli theorem, we deduce that for $\PP$-almost-every $\rho \in \Prob(\Omega)$, the sequence $(\bar{\phi}^n_\rho)_{n \geq 0}$ converges uniformly to $\phi^\infty_\rho$ in $\Class(\Omega)$ and that 
\begin{equation*}
    \nr{\phi^\infty_\rho}_{W^{1, \infty}(\Omega)} \leq 4R(1+R).
\end{equation*}
In particular, $(\phi^\infty_\rho)_\rho$ belongs to $\L^\infty(\PP; W^{1, \infty}(\Omega))$ and we have the limit
\begin{equation*}
     0 =\int_{\Prob(\Omega)} \bar{\phi}^n_\rho(\cdot) \dd \PP(\rho) \xrightarrow[n \to \infty]{} \int_{\Prob(\Omega)} \phi^\infty_\rho(\cdot) \dd \PP(\rho),
\end{equation*}
so that  $(\phi^\infty_\rho)_\rho$ is admissible to $\Dual'$. Eventually, for $\PP$-almost-every $\rho \in \Prob(\Omega)$, we have the limit
\begin{equation}
    \sca{(\bar{\phi}^n_\rho)^c}{\rho}  \xrightarrow[n \to \infty]{}  \sca{(\phi^\infty_\rho)^c}{\rho},
\end{equation}
so that by Lebesgue's dominated convergence theorem and the bound \eqref{eq:bar-phi-max-seq},
\begin{equation*}
    \int_{\Prob(\Omega)} \sca{(\phi^\infty_\rho)^c}{\rho}  \dd \PP(\rho) = \lim_{n \to +\infty} \int_{\Prob(\Omega)} \sca{(\bar{\phi}^n_\rho)^c}{\rho} \dd \PP(\rho) = \Dual',
\end{equation*}
which proves that $(\phi^\infty_\rho)_\rho \in \L^\infty(\PP; W^{1, \infty}(\Omega))$ is a maximizer for $\Dual'$.
\end{proof}

\section{Strong-convexity of \texorpdfstring{$\Kant_\rho$}{} for measures with non-convex support}
\label{sec:strong-convexity}

This section gathers occurrences of measures $\rho$ where the strong convexity estimate \eqref{it:ft-convexite} of Assumption \ref{assump:reg-rho} is verified. 

\subsection{Measures with convex support}
This result is mostly extracted from \cite{delalande:hal-03164147}. 
\begin{proposition}
\label{prop:strong-convexity-kanto-func}
Let $\rho \in \Prob_{a.c.}(\Omega)$. Assume that $\spt(\rho)$ is convex and that there exists $m_\rho, M_\rho \in (0, +\infty)$ such that $m_\rho \leq \rho \leq M_\rho$ on $\spt(\rho)$. Let $\psi, \tilde{\psi} \in \Class(\Omega)$. Then
\begin{align*}
    \sca{\psi - \tilde{\psi}}{(\nabla \psi^*)_\# \rho} + C_{d, R, m_\rho, M_\rho} \Var_{\rho}(\tilde{\psi}^* - \psi^*) \leq \Kant_\rho(\tilde{\psi}) - \Kant_\rho(\psi),
\end{align*}
where $C_{d,R, m_\rho, M_\rho} = \left( e(d+1)2^{d+1} R \diam(\spt(\rho)) \left( \frac{M_\rho}{m_\rho} \right)^2 \right)^{-1}$.
\end{proposition}
\begin{proof}
We only present here a formal sketch of the proof, which heavily relies on computations done in Section 2 of \cite{delalande:hal-03164147}. Assuming that $\psi$ and $\tilde{\psi}$ are smooth enough (see Proposition 2.4 of \cite{delalande:hal-03164147}) and introducing for $t \in [0,1], \psi^t = (1-t) \psi + t \tilde{\psi}$, Proposition 2.2 of \cite{delalande:hal-03164147} allows to differentiate $\Kant_\rho(\psi^t)$ with respect to $t$ and to obtain:
\begin{align}
\label{eq:formula-diff-K}
    \Kant_\rho(\tilde{\psi}) &- \Kant_\rho(\psi) = \frac{\dd}{\dd t} \Kant_\rho(\psi^t) \Big\vert_{t=0} + \int_0^1 \int_0^s \frac{\dd^2}{\dd t^2} \Kant_\rho(\psi^t) \dd t \dd s \notag \\
    &= \sca{\psi - \tilde{\psi}}{(\nabla \psi^*)_\# \rho} + \int_0^1 \int_0^s \int_\Omega \sca{\nabla v(\nabla (\psi^t)^*)}{\DD^2 (\psi^t)^* \cdot \nabla v(\nabla (\psi^t)^*)} \dd \rho \dd t \dd s,
\end{align}
were $v = \tilde{\psi} - \psi$. Reasoning as in the proof of Proposition 2.4 of \cite{delalande:hal-03164147}, the Brascamp-Lieb concentration inequality \cite{brascamp-lieb} and the log-concavity of the determinant seen as an application on the set of s.d.p. matrices ensure the following bound:
\begin{align*}
    C_{R, m_\rho, M_\rho} \min(t, 1-t)^d 2 \Var_{\frac{1}{2}(\mu + \tilde{\mu})}(\tilde{\psi} - \psi) \leq \int_\Omega \sca{\nabla v(\nabla (\psi^t)^*)}{\DD^2 (\psi^t)^* \cdot \nabla v(\nabla (\psi^t)^*)} \dd \rho,
\end{align*}
where $C_{R, m_\rho, M_\rho} = \left( e R \diam(\spt(\rho)) \left( \frac{M_\rho}{m_\rho} \right)^2 \right)^{-1}$, $\mu = (\nabla \psi^*)_\# \rho$ and $\tilde{\mu} = (\nabla \tilde{\psi})_\# \rho$. Back to \eqref{eq:formula-diff-K}, this leads to
\begin{align*}
    \sca{\psi - \tilde{\psi}}{(\nabla \psi^*)_\# \rho} + C_{d, R, m_\rho, M_\rho} 2 \Var_{\frac{1}{2}(\mu + \tilde{\mu})}(\tilde{\psi} - \psi) \leq \Kant_\rho(\tilde{\psi}) - \Kant_\rho(\psi),
\end{align*}
where $C_{d,R, m_\rho, M_\rho} = \left( e(d+1)2^{d+1} R \diam(\spt(\rho)) \left( \frac{M_\rho}{m_\rho} \right)^2 \right)^{-1}$. We conclude using the convex analysis argument of Proposition 3.1 from \cite{delalande:hal-03164147}, which directly ensures
$$ \Var_{\rho}(\tilde{\psi}^* - \psi^*) \leq 2 \Var_{\frac{1}{2}(\mu + \tilde{\mu})}(\tilde{\psi} - \psi). $$ We get the general case (without the smoothness assumptions on $\psi$ and $\tilde{\psi}$) using approximation arguments presented in Proposition 2.5 and 2.7 of \cite{delalande:hal-03164147}.
\end{proof}

\subsection{Measures with connected union of convex sets as support}
We extend Proposition \ref{prop:strong-convexity-kanto-func} to the case of a source measure $\rho$ with a possibly non-convex support.  We will assume that $\spt(\rho)$ can be written as a connected finite union of convex sets. 
\begin{proposition}
\label{prop:strong-convexity-kanto-func-2}
Let $\rho \in \Prob_{a.c.}(\Omega)$ such that there exists $m_\rho, M_\rho \in (0, +\infty)$ verifying $m_\rho \leq \rho \leq M_\rho$ on $\spt(\rho)$. Assume that $\spt(\rho)$ is connected and that there exists $N\geq 1$ convex sets $(C_i)_{1 \leq i \leq N}$ in $\Omega$ such that $\spt(\rho) = \bigcup_{i=1}^N C_i$. Also assume that for any $i \neq j$ such that $C_i \cap C_j \neq \emptyset$, one has $\rho(C_i \cap C_j) > 0$. Then there exists a constant $c_\rho$ depending on $\rho$ such that for any $\psi, \tilde{\psi} \in \Class(\Omega)$,
\begin{align*}
    \sca{\psi - \tilde{\psi}}{(\nabla \psi^*)_\# \rho} + c_\rho \Var_{\rho}(\tilde{\psi}^* - \psi^*) \leq \Kant_\rho(\tilde{\psi}) - \Kant_\rho(\psi).
\end{align*}
\end{proposition}

\begin{remark}[Constant $c_\rho$ and Poincaré-Wirtinger constant of $\rho$] 
The constant $c_\rho$ of Proposition \ref{prop:strong-convexity-kanto-func-2} is not made precise in the statement. A look at the proof of this proposition only allows to bound $c_\rho$ in terms of the second smallest eigenvalue $\lambda_2(L)$ of a weighted graph Laplacian $L$, that is built from the graph whose vertices are the convex sets $C_i$ and whose edge weights are the masses $\rho(C_i \cap C_j)$ that $\rho$ grants to the intersection of the convex sets $C_i$ and $C_j$. The constant $c_\rho$ then reads:
$$c_\rho = \left( e(d+1)2^{d+1} R^2 \left( \frac{M_\rho}{m_\rho} \right)^2  \left(N^2 + \frac{ 2 N^3}{\lambda_2(L)} \right)  \right)^{-1}.$$
The quantity $\lambda_2(L)$ is not explicit, but it can be linked to the \emph{weighted Cheeger constant} of $\rho$, defined by
\begin{equation*}
    h(\rho) = \inf_{A \subset \spt(\rho)} \frac{ \abs{\partial A}_\rho }{ \min(\rho(A), \rho(\spt(\rho) \setminus A)) },
\end{equation*}
where $\abs{\partial A}_\rho = \int_{\partial A \cap \interior(\spt(\rho))} \rho(x) \dd \Haus^{d-1}(x)$ and where the infimum is taken over Lipschitz domains $A \subset \interior(\spt(\rho))$ with boundary of finite $\Haus^{d-1}$-measure. Quoting \cite{Kitagawa2019ANA} (Lemma 5.3), this constant can in turn be linked to the $\L^1$ Poincaré-Wirtinger constant $C_{PW}(\rho)$ of $\rho$. Indeed, $h(\rho)$ is positive whenever $\rho$ satisfies an $\L^1$ Poincaré-Wirtinger inequality, i.e. whenever there exists a finite $C_{PW}(\rho) > 0$ such that for all smooth function $f$ on $\Omega$,
$$ \nr{f - \Esp_\rho f}_{\L^1(\rho)} \leq C_{PW}(\rho) \nr{\nabla f}_{\L^1(\rho; \Rsp^d)}. $$
The Poincaré-Wirtinger constant $C_{PW}(\rho)$ and the Cheeger constant $h(\rho)$ are then related by the inequality
\begin{equation*}
    h(\rho) \geq \frac{2}{C_{PW}(\rho)}.
\end{equation*}
Using ideas similar to the ones found in Section 5.2 of \cite{Kitagawa2019ANA}, the eigenvalue $\lambda_2(L)$ can be bounded in terms of the Cheeger constant of $\rho$, and thus in terms of $C_{PW}(\rho)$. We do not detail this comparison here but only report that $c_\rho$ may be written
$$c_\rho =  \left( e(d+1)2^{d+1} R^2 \left( \frac{M_\rho}{m_\rho} \right)^2 N \left(N + \frac{1}{2}\left(\frac{ M_\rho s_{d-1} R^{d-1} N^2 C_{PW}(\rho)}{ \eps^2  }\right)^3 \right) \right)^{-1},$$
where $s_{d-1}$ denotes the surface area of the unit sphere in $\Rsp^d$ and
\begin{equation*}
    \eps= \min \left( \min_{i, j \vert C_i \cap C_j \neq \emptyset} \rho(C_i \cap C_j), \min_i \rho\left(C_i \setminus \cup_{j \neq i} C_j \right) \right) > 0.
\end{equation*} 
\end{remark}

\begin{proof}[Proof of Proposition~\ref{prop:strong-convexity-kanto-func-2}]
Let's denote for now $f = \tilde{\psi}^* - \psi^*$. We will first exploit a discrete Laplacian over $\X = \spt(\rho)$ in order to upper bound $\Var_\rho(f)$ by a sum of variances of $f$ w.r.t. probability measures supported over the convex sets $(C_i)_i$. We will then use Proposition \ref{prop:strong-convexity-kanto-func} to conclude.

For any $i \in \{1,\dots,N\}$, we denote $\rho_i = \frac{1}{\rho(C_i)} \rho_{\vert C_i}$ and $m_i = \int_{C_i} f \dd \rho_i$.
Then one has the following bound:
\begin{align}
\label{eq:bound-var-sum-sum-var}
    \Var_\rho(f) &= \frac{1}{2} \int_{\X \times \X} (f(x) - f(y))^2 \dd \rho(x) \dd \rho(y) \notag \\
    &\leq\frac{1}{2} \sum_{i,j} \int_{C_i \times C_j} (f(x) - f(y))^2 \dd \rho(x) \dd \rho(y) \notag \\
    &= \frac{1}{2} \sum_{i,j} \int_{C_i \times C_j} (f(x) - m_i + m_i - m_j + m_j - f(y))^2 \dd \rho(x) \dd \rho(y) \notag \\
    &= \left(\sum_i \rho(C_i) \right) \sum_i \int_{C_i} (f(x) - m_i)^2 \dd \rho(x) + \frac{1}{2} \sum_{i,j} (m_i - m_j)^2 \rho(C_i) \rho(C_j) \notag \\
    &= \left(\sum_i \rho(C_i) \right) \sum_i \rho(C_i) \Var_{\rho_i}(f) + \frac{1}{2} \sum_{i,j} (m_i - m_j)^2 \rho(C_i) \rho(C_j).
\end{align}

We now consider the graph $G = (\{C_i\}_{1 \leq i \leq N}, \{w_{ij}\}_{1 \leq i,j \leq N})$ with vertices $\{C_i\}_{1 \leq i \leq N}$ and weighted edges $\{w_{ij}\}_{1 \leq i,j \leq N}$ defined by
\begin{equation*}
    \forall i,j \in \{1, \dots, N\}, \quad w_{ij} = \rho(C_i \cap C_j).
\end{equation*}
By construction, this graph has a single connected component. We introduce the weighted Laplacian matrix $L \in \Rsp^{N\times N}$ of $G$ as follows:
\begin{equation*}
    \forall i,j \in \{1, \dots, N\}, \quad L_{ij} = \left\{
    \begin{array}{ll}
        \sum_{k} w_{ik} & \mbox{if } i = j, \\
        -w_{ij} & \mbox{else.}
    \end{array}
\right.
\end{equation*}
Then $L$ is a symmetric and positive semi-definite matrix. Its null space is made of constant vectors and we denote $\lambda_2(L)$ its second smallest eigenvalue, which is non-zero. Denoting $m = (m_i)_{1\leq i \leq N} \in \Rsp^N$, we introduce $\bar{m} = \left(\frac{1}{N} \sum_i m_i\right) \mathds{1}_N \in \Rsp^N$ the constant vector whose coordinates equal the mean of $m$ (we use $\mathds{1}_N = (1)_{1\leq i \leq N} \in \Rsp^N$). Notice that $m - \bar{m}$ is in the orthogonal to the null space of $L$, ensuring the following bound:
\begin{align}
\label{eq:bound-laplacian-1}
    \frac{1}{2} \sum_{i,j} (m_i - m_j)^2 \rho(C_i) \rho(C_j) &\leq N^2 \frac{1}{2} \sum_{i,j} (m_i - m_j)^2 \frac{1}{N^2} \notag \\
    &=  N \nr{m - \bar{m} }^2 \notag \\
    &\leq \frac{ N  }{\lambda_2(L)} \sca{m - \bar{m}}{L \left(m - \bar{m}\right)} \notag \\
    &= \frac{  N  }{\lambda_2(L)}  \sum_{i, j} w_{ij} (m_i^2 - m_i m_j) \notag \\
    &= \frac{  N  }{\lambda_2(L)} \sum_{i, j} \frac{w_{ij}}{2} (m_i - m_j)^2.
\end{align}
But for any $i,j$ such that $w_{ij}>0$, denoting $m_{i \cap j} = \frac{1}{\rho(C_i \cap C_j)} \int_{C_i \cap C_j} f \dd \rho$, one has
\begin{align*}
    \frac{1}{2} (m_i - m_j)^2 \leq (m_{i \cap j} - m_i)^2 + (m_{i \cap j} - m_j)^2.
\end{align*}
And for such $i,j$,
\begin{align*}
    (m_{i \cap j} - m_i)^2 &= \left( \frac{1}{\rho(C_i \cap C_j)} \int_{C_i \cap C_j} (f - m_i) \dd \rho \right)^2 \\
    &\leq \frac{1}{\rho(C_i \cap C_j)} \int_{C_i} (f - m_i)^2 \dd \rho \\
    &= \frac{\rho(C_i)}{w_{ij}} \Var_{\rho_i}(f),
\end{align*}
where we used Jensen's inequality and the fact that $C_i \cap C_j \subset C_i$. A similar bound can be shown for $(m_{i \cap j} - m_j)^2$, and plugging these into \eqref{eq:bound-laplacian-1} yields
\begin{align*}
    \frac{1}{2} \sum_{i,j} (m_i - m_j)^2 \rho(C_i) \rho(C_j) &\leq \frac{  N  }{\lambda_2(L)}  \sum_{i} \sum_{j \vert C_i \cap C_j \neq \emptyset } \left(\rho(C_i) \Var_{\rho_i}(f) + \rho(C_j) \Var_{\rho_j}(f)\right) \\
    &\leq \frac{  2 N^2  }{\lambda_2(L)}  \sum_{i} \rho(C_i) \Var_{\rho_i}(f).
\end{align*}
Injecting this into \eqref{eq:bound-var-sum-sum-var} yields
\begin{align}
\label{eq:bound-var-sum-sum-var-2}
    \Var_\rho(f) \leq \left(N + \frac{ 2 N^2}{\lambda_2(L)} \right) \sum_i \rho(C_i) \Var_{\rho_i}(f).
\end{align}
Now recalling that $f = \psi - \tilde{\psi}$, we have by Proposition \ref{prop:strong-convexity-kanto-func} for any $i\in \{1, \dots, N\}$ that
\begin{align*}
    \sca{\psi - \tilde{\psi}}{(\nabla \psi^*)_\# \rho_i} + C_{d,R, m_\rho, M_\rho} \Var_{\rho_i}(\tilde{\psi}^* - \psi^*) \leq \Kant_{\rho_i}(\tilde{\psi}) - \Kant_{\rho_i}(\psi),
\end{align*}
where $C_{d,R, m_\rho, M_\rho}= \left( e(d+1)2^{d+1} R^2 \left( \frac{M_\rho}{m_\rho} \right)^2 \right)^{-1} $. Weighting this last inequality with $\rho(C_i)$ and summing over $i \in \{1, \dots, N\}$, this raises
\begin{align*}
    \sca{\psi - \tilde{\psi}}{(\nabla \psi^*)_\# \rho} + \frac{C_{d,R, m_\rho, M_\rho}}{N} \sum_{i=1}^N \rho(C_i) \Var_{\rho_i}(\tilde{\psi}^* - \psi^*) \leq \Kant_{\rho}(\tilde{\psi}) - \Kant_{\rho}(\psi).
\end{align*}
Using \eqref{eq:bound-var-sum-sum-var-2} eventually gives
\begin{align*}
    \sca{\psi - \tilde{\psi}}{(\nabla \psi^*)_\# \rho} + c_{\rho} \Var_{\rho}(\tilde{\psi}^* - \psi^*) \leq \Kant_{\rho}(\tilde{\psi}) - \Kant_{\rho}(\psi),
\end{align*}
where $c_\rho = \left( e(d+1)2^{d+1} R^2 \left( \frac{M_\rho}{m_\rho} \right)^2  \left(N^2 + \frac{ 2 N^3}{\lambda_2(L)} \right)  \right)^{-1}$.
\end{proof}

\bibliographystyle{plain}
\bibliography{ref}

\end{document}